\documentclass[reqno]{amsart}

\usepackage{hyperref}
\usepackage{mathrsfs}
\usepackage{amssymb}
\usepackage{latexsym}
\usepackage[mathscr]{eucal}
\usepackage{slashed}

\usepackage[all]{xy}
\usepackage{pb-diagram,pb-xy}
\usepackage{setspace}

\usepackage[a4paper, hmargin=3.0cm, tmargin=3.2cm, bmargin=3.2cm]{geometry}
\numberwithin{equation}{section}

\newtheorem{MainThm}{Pretheorem}

\theoremstyle{definition}

\newtheorem{defn}[equation]{Definition}

\newtheorem{remark}[equation]{Remark}
\newtheorem{example}[equation]{Example}

\theoremstyle{plain}
\newtheorem{lem}[equation]{Lemma}
\newtheorem{thm}[equation]{Theorem}

\newtheorem{prop}[equation]{Proposition}
\newtheorem{corollary}[equation]{Corollary}
\newcommand{\Mfds}{\mathbf{Mfds}}
\newcommand{\Set}{\mathbf{Set}}
\newcommand{\sSet}{\mathbf{sSet}}
\newcommand{\Top}{\mathbf{Top}}
\newcommand{\Sheaf}{\mathbf{Sheaves}}
\newcommand{\Cob}{\mathbf{Cob}}

\newcommand{\sh}{\mathrm{sh}}
\newcommand{\Ob}{\mathrm{Ob}}
\newcommand{\Sing}{\mathrm{Sing}}

\newcommand{\Th}{\mathrm{Th}}
\newcommand{\MT}{\mathrm{MT}}

\newcommand{\GRW}{\mathrm{GRW}}

\newcommand{\loopinf}[1]{\Omega^{\infty #1}}

\newcommand{\bD}{\mathbb{D}}
\newcommand{\bN}{\mathbb{N}}
\newcommand{\bR}{\mathbb{R}}
\newcommand{\bS}{\mathbb{S}}
\newcommand{\bZ}{\mathbb{Z}}
\newcommand{\bK}{\mathbb{K}}

\newcommand{\gA}{\mathbf{A}}

\newcommand{\ev}{\mathrm{ev}}

\newcommand{\cK}{\mathbb{K}}
\newcommand{\cF}{\mathcal{F}}

\newcommand{\cC}{\mathcal{C}}
\newcommand{\cD}{\mathcal{D}}
\newcommand{\cB}{\mathcal{B}}
\newcommand{\cG}{\mathcal{G}}
\newcommand{\cS}{\mathcal{S}}
\newcommand{\cL}{\mathcal{L}}

\newcommand{\bott}{\operatorname{bott}}

\newcommand{\cTh}{\mathcal{T}}

\newcommand{\thom}{\operatorname{thom}}

\newcommand{\Lin}{\mathbf{Lin}}

\newcommand{\Cl}{\mathbf{Cl}}
\newcommand{\colim}{\operatorname{colim}}

\newcommand{\conn}{\operatorname{conn}}

\newcommand{\ins}{\operatorname{ins}}

\newcommand{\scan}{\mathrm{scan}}

\newcommand{\hocolim}{\mathrm{hocolim}}
\newcommand{\Diff}{\mathrm{Diff}}

\newcommand{\eps}{\epsilon}

\newcommand{\End}{\operatorname{End}}

\newcommand{\im}{\operatorname{Im}}
\newcommand{\ind}{\operatorname{index}}
\newcommand{\topind}{\operatorname{topind}}
\newcommand{\id}{\operatorname{id}}

\newcommand{\Kom}{\mathbf{Kom}}

\newcommand{\normalize}[1]{\frac{#1}{(1+{#1}^{2})^{1/2}}}
\newcommand{\pr}{\operatorname{pr}}

\newcommand{\gC}{\mathbf{C}}
\newcommand{\gx}{\mathbf{x}}

\newcommand{\gy}{\mathbf{y}}
\newcommand{\gv}{\mathbf{v}}
\newcommand{\Gr}{\mathrm{Gr}}

\newcommand{\PT}{\mathrm{PT}}

\newcommand{\spec}{\operatorname{spec}}
\newcommand{\Spin}{\mathrm{Spin}}

\newcommand{\scpr}[1]{\langle #1 \rangle}

\newcommand{\ad}{\mathrm{ad}}

\newcommand{\cstar}{\mathrm{C}^{\ast}}

\newcommand{\Dir}{\slashed{\mathfrak{D}}}
\newcommand{\spinor}{\slashed{\mathfrak{S}}}

\newcommand{\smb}{\mathrm{smb}}

\newcommand{\SO}{\mathrm{SO}}

\newcommand{\norm}[1]{\| #1 \|}

\newcommand{\KO}{KO}

\newcommand{\Vect}{\mathcal{V}}

\title[Index theory in spaces of manifolds]{Index theory in spaces of manifolds}

\author{Johannes Ebert}
\thanks{Partially supported by the SFB 878 ``Groups, Geometry and Actions''}
\address{Mathematisches Institut, Universit\"at M\"unster\\
Einsteinstra{\ss}e 62\\
48149 M\"unster\\
Bundesrepublik Deutschland}

\email{johannes.ebert@uni-muenster.de}
\date{\today}


\begin{document}

\begin{abstract}
We formulate and prove a generalization of the Atiyah-Singer family index theorem in the context of the theory of spaces of manifolds \`a la Madsen, Tillmann, Weiss, Galatius and Randal-Williams. Our results are for Dirac-type operators linear over arbitrary $C^*$-algebras.
\end{abstract}

\maketitle

\tableofcontents

\section{Introduction}

The historically first proof of the Atiyah-Singer index theorem \cite{AS0}, \cite{Pal} established an intimate relation between cobordism theory and index theory of elliptic operators. It relied on the cobordism invariance of the index and Thom's computation of the rational oriented cobordism ring; the better known $K$-theoretic proof \cite{ASIoeoI,ASIoeoIV} eliminated the dependence on cobordism theory.

During the last 15 years, we have witnessed a revival of the geometric aspects of cobordism theory, starting from \cite{MW}, made more explicit in \cite{GMTW} and further developed in \cite{GRW10}. 
This geometric theory concerns the $d$-dimensional \emph{cobordism category of $\theta$-manifolds} $\Cob^{\theta} (d)$, where $\theta:Y \to BO(d)$ is a fibration. Galatius, Madsen, Tillmann and Weiss proved in \cite{GMTW} that $B \Cob^{\theta} (d) \simeq \Omega^{\infty-1} \MT \theta(d)$, where $\MT \theta(d)$ is the Thom spectrum of the additive inverse of the vector bundle classified by $\theta$. For a closed $d$-manifold $M$, there is a map $\alpha:B\Diff_\theta (M)\to \Omega B \Cob^{\theta}(d)\simeq \Omega^{\infty} \MT \theta(d)$ from the classifying space for $M$-bundles with $\theta$-structure. 

For some $\theta$, there are natural elliptic operators living on $\theta$-manifolds (for example, the Cauchy-Riemann operator on Riemann surfaces, the signature operator on oriented manifolds and the spin Dirac operator on spin manifolds).
It was observed no later than \cite{MT} that some of the homotopy theoretic constructions around the spectrum $\MT \SO(2)$ admit interpretations in terms of the index of the Cauchy-Riemann operator, as a consequence of the Atiyah-Singer theorem. This observation was later systematized by the author \cite{EbertV} and used many times, e.g. in \cite{BERW}. In each of these situations, there is a spectrum map $\MT \theta \to \bK$ to the $K$-theory spectrum or some (de)suspension thereof, defined using the underlying linear algebra. The index theorem implies that the composition of $\alpha$ with that map is homotopic to the classifying map for the family index of the operators under consideration. 

The main result of \cite{GMTW} provides a geometric representation of the space $\loopinf{-1} \MT \theta(d)$ in terms of spaces of $\theta$-manifolds. This suggests the possibility of finding a \emph{proof} of the index theorem using the techniques introduced in \cite{GMTW}, and the purpose of this paper is to present such a proof. At the same time, we give a generalization of the index theorem to families of noncompact manifolds. Let us describe the idea. 

For sake of concreteness, suppose we wish to compute the family index of the spin Dirac operator $\Dir$ on a bundle of $d$-dimensional closed spin manifolds $\pi:M \to X$. The Dirac operator is linear over the Clifford algebra $\Cl^{d,0}$, and so we expect it to have an index\footnote{We denote the $K$-theory spectrum of a Real graded $\cstar$-algebra $\gA$ by $\bK(\gA)$ and the $n$th space in this spectrum by $\bK (\gA)_n$. Hence $[X; \bK (\Cl^{d,0})_n] \cong KO^{n-d}(X)$.} $\ind(\Dir)\in KO^{-d} (X) = [X; \bK (\Cl^{d,0})_0]$. 

At the heart of the new geometric cobordism theory, there are two spectra (in the sense of homotopy theory) $\MT\Spin (d)$ and $\GRW\Spin (d)$. The spectrum $\MT\Spin (d)$ is nowadays quite well-known, so let us focus on the other one, which was introduced (with a different name) by Galatius and Randal-Williams \cite{GRW10}. The $0$th space $\GRW \Spin (d)_0$ of $\GRW \Spin (d)$ is the space of all closed $d$-dimensional spin manifolds. There is a suitable topology on $\GRW \Spin (d)_0$ \cite{GRW10}, and with this topology, $\GRW \Spin (d)_0$ becomes a classifying space for fibre bundles of $d$-dimensional closed spin manifolds. Hence the bundle $\pi$ corresponds to a map $\lambda_{\pi}:X \to \GRW \Spin (d)_0$, unique up to homotopy. 
The usual stability properties of the Fredholm index imply that $\ind (\Dir) \in KO^{-d} (X)$ only depends on the homotopy class of $\lambda_{\pi}$ (and not on data such as fibrewise Riemannian metrics which enter the definition of $\Dir$). 
One can go a step further, and define a universal index map $\ind_0:\GRW \Spin (d)_0 \to  \bK (\Cl^{d,0}) \simeq \loopinf{+d} \KO$ in terms of analysis. The goal of the index theorem is to obtain a topological formula for $\ind_0$.

The main idea of the present paper is to extend $\ind_0$ to a spectrum map $\GRW \Spin (d) \to \bK (\Cl^{d,0})$ and to take advantage of the results of \cite{GRW10} to compute it in terms of homotopy theory. 
A point in the $n$th space $\GRW\Spin (d)_n$ of the spectrum $\GRW \Spin (d)$ is a noncompact $d$-dimensional spin manifold $M$, equipped with a proper ``control map'' $f: M \to \bR^n$, and a Riemannian metric. Pretending for a moment that $M $ is complete (which is not the case in general), the Dirac operator on such an $M$ is essentially self-adjoint, and we may form the bounded transform $F:=\normalize{\Dir}$. However, unless $n=0$ (and hence $M$ compact), $F$ will not be a Fredholm operator, and hence does not have an index in $K (\Cl^{d,0})$. But there is an index, which lives in another $K$-group. To see which one, we take guidance from Kasparov theory (even though in the end our results are formulated and proven without referring to Kasparov theory).

An important result in the analysis of elliptic operators on noncompact manifolds (see e.g. \cite[\S 10]{HR}) states that for each compactly supported function $h $ on $\bR^n$, the operator $(h\circ f)(F^2 -1)$ is compact. That is, $\Dir$ defines a class in the Kasparov group $KK(\gC_0 (\bR^n); \Cl^{d,0})\cong KO^{n-d}(*)$. 
Instead of this group, we shall use an isomorphic group to store the information about the operators on $M$, namely $KK (\Cl^{n,0},\Cl^{d,0})$. The isomorphism $KK(\gC_0 (\bR^n); \Cl^{d,0}) \cong KK (\Cl^{n,0},\Cl^{d,0})$ is given abstractly by an intersection product, but we can give a very concrete and simple description of the image of the class of $\Dir$, using a kind of ``dual Dirac'' element. We replace $\Dir$ by an operator $\Dir'$ with an extra Clifford symmetry and compact resolvent, so that $\normalize{\Dir'}$ is Fredholm. 
By Kasparov's Bott periodicity theorem, we know that this process does not loose index-theoretic information. The construction can be carried out in the parametrized setting, and we obtain index maps 
\begin{equation}\label{intro:controlled-index}
\ind_n : \GRW \Spin (d)_n \to \bK (\Cl^{d,0})_n \simeq \loopinf{-n} \bK (\Cl^{d,0}).
\end{equation}
The construction relies on the generalization of the classical regularity theory for elliptic operators which the author developed in \cite{JEIndex1} (to use these analytical results, we have to replace the source of \eqref{intro:controlled-index} by a homotopy equivalent space, but let us ignore this techical point for now). 
Now both, target and source are the $n$th space of spectra. On the right hand side, the structure maps are given by the Bott maps (or appropriate versions thereof). On the left-hand side, the structure map is a fairly tautological construction (sometimes called ``scanning map''), which might be described as follows. Let $M$ be a manifold with a proper control map $f: M \to \bR^n$. For each $t \in \bR$, we get a new control map $(f,t): M \to \bR^n \times \bR$. As $t$ runs from $-\infty$ to $+\infty$, we get a family of manifolds (all equal to $M$) and control maps, namely $(f,t)$. The topology on $\GRW \Spin (d)_{n+1}$ is designed in such a fashion that this family can be completed at $\pm \infty$ by adding the empty manifold. The construction of this scanning map, the Bott map and the index fit together so that the collection $(\ind_n)_n$ is a map of spectra
\begin{equation}\label{intro:controlled-index2}
\ind : \GRW \Spin (d) \to \bK (\Cl^{d,0}).
\end{equation}
(not quite: it is only a weak map in the sense of \ref{lem:strictification-of-spectrummaps} below; the reason is that certain canonical isomorphisms are not identities). 

The space $\MT \Spin (d)_n$ can be viewed as a subspace of $\GRW \Spin(d)_n$, namely the space of all linear submanifolds contained in $\bR^n$ (a linear submanifold is a, possibly empty, affine subspace). The inclusion maps $\MT \Spin (d)_n \to \GRW\Spin(d)_n$ together give a map of spectra. The key result about this map is that it is a stable equivalence of spectra \cite{GRW10}. This reduces the computation of the spectrum map \eqref{intro:controlled-index2} to the much smaller spectrum $\MT\Spin(d)$. This is a fairly straightforward task, using the Thom isomorphism theorem in $K$-theory and the computation of the spectrum of the supersymmetric harmonic oscillator.

We will not only prove an index theorem for the spin Dirac operator, but for all other operators of Dirac type, and they are allowed to be linear over arbitrary, possibly graded and Real $C^*$-algebras $\gA$ (for example group $C^*$-algebras).
In that case, we have to replace $\GRW \Spin (d)$ by a spectrum $\GRW \theta_{\gA} (d)$; a point in the $n$th space is a triple $(M,f,E)$, with $M$ a manifold, $f: M \to \bR^n$ a proper map and $E \to M$ a bundle of graded finitely generated projective $\gA$-modules, together with a $\Cl (TM)$-structure. This spectrum fits into the general framework of \cite{GRW10}, in particular, there is a Thom spectrum $\MT \theta_{\gA} (d)$ and a weak equivalence of spectra $\Lambda:\MT \theta_{\gA} (d) \to \GRW \theta_{\gA} (d)$. The $K$-theory spectrum $\bK(\Cl^{d,0})$ is replaced by $\bK(\gA)$, the $K$-theory spectrum of the graded $C^*$-algebra $\gA$.

Let us now formulate the main results of this paper in rough terms (compare \cite[p. 45]{Adams} for our usage of the word ``pretheorem''). 
The spectrum $\MT \theta_\gA(d)$ is a Thom spectrum, and there is a Thom class which is a (weak) map of spectra
\[
 \topind: \MT \theta_\gA (d) \to \bK(\gA),
\]
the \emph{topological index}. 

\begin{MainThm}[Precise statement given in Proposition \ref{analytical-index-vs-scanning} and Theorem \ref{thm:index-theorem}]\label{main-theorem}
For each graded Real $C^*$-algebra $\gA$, there is a weak spectrum map 
\[
\ind: \GRW \theta_{\gA} (d) \to \bK \gA.
\]
On the $0$th space, it classifies the ordinary family index of $\gA$-linear Dirac operators. 
The composition of $\ind$ with the natural equivalence $\Lambda: \MT \theta_{\gA} (d) \to \GRW \theta_{\gA} (d)$ is homotopic (as weak maps of spectra) to $\topind$. 
\end{MainThm}

There is a corollary of Theorem \ref{main-theorem} which looks more closely related to the classical index theorem. We define a map
\[
 \PT_n: \GRW \theta_\gA(d)_n \stackrel{\tau_n}{\to} \Omega^{\infty-n}\GRW \theta_\gA(d) \stackrel{p_n}{\to} \Omega^{\infty-n} \MT \gA(d)
\]
as the composition of the map $\tau_n$ given by the spectrum structure and a homotopy inverse $p_n$ to the homotopy equivalence $ \Omega^{\infty-n} \Lambda : \Omega^{\infty-n} \MT \theta_{\gA}(d) \to \Omega^{\infty-n} \GRW\theta_\gA(d)$ (this map can also be constructed by a parametrized Pontrjagin-Thom construction). 
\begin{MainThm}[Precise statement given in Corollary \ref{cor:indextheorem-space-version}]\label{main-cor}
In the situation of Theorem \ref{main-theorem}, the two maps
\[
( \GRW\theta_\gA)_n \to \bK(\gA)_n \simeq \Omega^{\infty-n} \bK(\gA)
\]
given by $\ind_n$ and $(\Omega^{\infty-n}\topind) \circ PT_n$ are homotopic.
\end{MainThm}

\begin{remark}
The classical formulation of the index theorem for real operators involves Atiyah's $KR$-theory. In this paper, there is no $KR$-theory. This is possible since we only consider operators of Dirac type, and for those, the appearance of $KR$-theory can be eliminated, at the expense of introducing a mildly twisted version of $K$-theory. Let us explain this in the simplest situation. Let $M^d\subset \bR^n$ be a closed manifold and $E \to M$ a real $\Cl(TM)$-bundle. The symbol class $\sigma(D)$ of the Dirac operator $D$ on $E$ is an element in $KR_c (TM^-)$, and $E$ itself defines an element $[E] \in K^{TM} (M)$ (see Definition \ref{defn:twistedKgroup} below for the definition of this twisted $K$-group), which maps to $\sigma(D)$ under the Thom isomorphism $K^{TM}(M) \to KR_c (TM^-)$. The classical index theorems can be stated by saying that $[E]$ maps to $\ind (D)$ under the composition
\[
K^{TM}(M) \cong K^{TM^- \oplus NM \oplus NM^-} (M) \cong KO^n_c (NM) \to KO^n_n (\bR^n) \cong \bZ
\]
of a Morita equivalence isomorphism, the Thom isomorphism, the pushforward along open embeddings and the Bott periodicity isomorphism. In this formulation, no $KR$-group shows up explicitly.
If one would like to prove an index theorem for more general operators than Dirac operators (e.g. pseudo-differential operators), this trick would not be available. There are also analytical difficulties with treating more general operators, and we refrain from considering them.
\end{remark}

\begin{remark}
If one allows arbitrary coefficient $C^*$-algebras $\gA$, our index theorem provides generalizations of the classical results by Mishchenko  and Fomenko \cite{MF}. In particular, the present paper proves a family version of the index theorem of \cite{MF}, even for graded $C^*$-algebras. Even though this is certainly an expected result, it does not seem to be documented in the literature. It could be proven using Kasparov's $KK$-theory, following the line of argument by Connes and Skandalis \cite{CoSk}. 
\end{remark}

\begin{remark}
Let us explain the meaning of $\ind_n$ in a simple situation (taking as an example the spin Dirac operator). A point in $\GRW \Spin(d)_n$ is a pair $(M,f)$, consisting of a $d$-dimensional noncompact spin manifold $M$, and a proper smooth map $f: M \to \bR^n$. Then $\ind_n (M,f) \in \bK(\Cl^{d,0})_n$, and let $\mathrm{ind}_n (M,f) \in \pi_0 (\bK(\Cl^{d,0})_n) \cong KO_{n-d}(*)$ be the path component of $\ind_n (M,f)$. To compute this element, choose a regular value $a \in \bR^n$ of $f$ and let $N:= f^{-1}(a)$, which is a closed $(d-n)$-dimensional spin manifold. Let $\pr_2:N \times \bR^n \to \bR^n$ be the projection map. The topology of $\GRW\Spin (d)_n$ is designed in such a way that $(M,f)$ and $(N \times \bR^n,\pr_2)$ lie in the same component of $\GRW \Spin (d)_n$. Therefore $\mathrm{ind}_n (M,f) = \mathrm{ind}_n (N \times \bR^n,\pr_2)$. One can show that $\mathrm{ind}_n (N \times \bR^n,\pr_2)$ is the index of the spin Dirac operator on the closed manifold $N$.
For $n=1$, this can be interpreted as an instance of the ``partitioned manifold index theorem'' of Roe \cite[Theorem 3.3]{Roe-partitioned}, see also \cite[Theorem 1.5]{Higson-cobordism}. 
We are not aware of a simple description of $\ind_n$ in the parametrized situation.
\end{remark}

{\bf Outline of the paper.} The purpose of section \ref{sec:background} is to gather the topological results we need. When dealing with spaces of manifolds, it is convenient to use the abstract sheaf-theoretic language used by Madsen and Weiss \cite{MW}, so we recall this in subsection \ref{subsec:homotopytheory}. We then proceed to survey results of Galatius and Randal-Williams from \cite{GRW10} and put them into the form we need (stated as Theorem \ref{grwtheorem2} and \ref{grwtheorem1} below). 
In section \ref{sec:construction-index-amp}, we construct the spectrum map $\ind: \GRW \theta_{\gA} \to \bK \gA$. The analytical work in \cite{JEIndex1} was carried out with that goal in mind, so that the construction is pretty straightforward. 
Section \ref{sec:index-theorem} contains the proof of Theorem \ref{main-theorem}. 

\section{Background material}\label{sec:background}

\subsection{The language of sheaves}\label{subsec:homotopytheory}

The results of this paper involve spaces whose points are manifolds (equipped with extra data, such as Dirac operators). In \cite{GRW10}, a topology on such spaces is constructed. For our purposes, it is more convenient to avoid delicate questions in point-set topology by following the functor-of-points-philosophy. More precisely, we shall use the formalism of sheaves as in \cite[\S 2.1, 2.4]{MW}, which we now briefly recall. 

Let $\Mfds$ be the category of smooth manifolds and smooth maps, referred to as \emph{test manifolds}. A \emph{sheaf} is a contravariant functor $\cF:\Mfds \to \Set$ which satisfies the usual gluing condition. That is, if $(U_i)_{i \in I}$ is an open cover of a test manifold $X$ and if $z_i \in \cF (U_i)$ are elements such that for each pair $(i,j) \in I^2$ we have\footnote{We denote the pullback along inclusions by the restriction symbol.} $z_i |_{U_i \cap U_j} = z_j|_{U_i \cap U_j}$, then there is a unique $z \in \cF (X)$ with $z|_{U_i} =z_i$.

One might think of $\cF$ as a space whose points are the elements of $\cF(*)$, and elements $z \in \cF(X)$ induce continuous maps $X \to \cF(*)$, $x \mapsto z|_{\{x\}}$. To get a grasp of the definitions/statements/argu\-ments that follow, we advise the reader to secretly put $X=*$ on the first reading.

Sheaves on $\Mfds$ form a category $\Sheaf$, and there is a functor $\Sheaf\to \sSet$ to the category of simplicial sets, defined as follows. Let $\Delta^p_e := \{ x \in \bR^{p+1} \vert \sum_{i=0}^{p} x_i =1\}$ be the ``extended $p$-simplex''. Then $p \mapsto \cF (\Delta^p_e)$ is a simplicial set, denoted $\cF_{\bullet}$. The \emph{representing space of $\cF$} is by definition the geometric realization $|\cF|:= | \cF_{\bullet}|$ of this simplicial set.

A topological space $Y \in \Ob (\Top)$ defines a sheaf $\sh (Y)$, namely $\sh (Y)(X):= \{f: X \to Y \vert f\text{ continuous }\}$. The simplicial set $\sh (Y)_{\bullet}$ is the extended singular simplicial set $\Sing^e_{\bullet} Y$ of $Y$, defined using extended simplices. There is an obvious map $\Sing^e_{\bullet} Y \to \Sing_{\bullet} Y$ which is a weak equivalence of simplicial sets. 

A \emph{concordance} between two elements $z_0,z_1 \in \cF (X)$ is an element $z \in \cF (X \times \bR)$ such that $z|_{X \times \{i\}}=z_i$ for $i=0,1$. Concordance is an equivalence relation, and the set of concordance classes is denoted $\cF[X]$. 
It is proven in \cite[Proposition 2.17]{MW} that there is a natural bijection 
\begin{equation}\label{eqn:sheaf-representing}
\cF[X] \cong [X; |\cF|], 
\end{equation}
for each $X \in \Mfds$. We say that a map $\cF \to \cG$ of sheaves is $n$-connected (or a weak equivalence) if the induced map $|\cF| \to |\cG|$ is $n$-connected (or a homotopy equivalence). 

Let $\cF^{\bR}$ be the sheaf $\cF^{\bR} (X):= \cF(X \times \bR)$. It comes with evaluation maps $\ev_t: \cF^\bR \to \cF$, $z \mapsto z|_{X \times \{t\}}$. A \emph{homotopy} or \emph{natural concordance} between two maps $F_0,F_1: \cF \to \cG$ of sheaves is a map $F: \cF \to \cG^\bR$ such that $\ev_i \circ F=F_i$. 

A \emph{basepoint} of a sheaf $\cF$ is a morphism $z_0:* \to \cF$ from the initial sheaf (this is the same information as a consistent choice of basepoints of the sets $\cF(X)$). If $(\cF,z_0)$ is a pointed sheaf, we define the \emph{loop sheaf} $\Omega_{z_0} \cF$ as follows: $\Omega_{z_0} \cF(X) $ is the set of all $z \in \cF ( X \times \bR)$ with $z|_{X\times \{i\}}=z_0$ for $i=0,1$. If the basepoint $z_0$ is understood, it is dropped from the notation. There is a map of simplicial sets
\[
 \phi_\bullet: \Delta^1_\bullet \times (\Omega \cF)_\bullet \to \cF_\bullet
\]
defined by $(\alpha,z) \mapsto \tilde{\alpha}^* z$. To understand the notation, let $\alpha \in \Delta^1_p$ be a $p$-simplex. It induces an affine map $\Delta^p_e \to \Delta^1_e$, and so $\tilde{\alpha}: \Delta_e^p \to \Delta_e^1 \times \Delta_e^p = \bR \times \Delta_e^p$. The geometric realization of $\phi_\bullet$ is a map $\Delta^1 \times |\Omega \cF|\to |\cF|$ which descends to a pointed map $S^1 \wedge |\Omega \cF| \to |\cF|$, whose adjoint is a map
\begin{equation}\label{map-loop-space-sheafworlojjhs}
 \Phi: |\Omega \cF| \to \Omega |\cF|.
\end{equation}
Using \cite[Proposition 2.17]{MW}, one shows that $\Phi$ is a homotopy equivalence. 

\begin{remark}\label{rem:chage-coordinates-for-loopsheaf}
It is useful for us to change coordinates in the $\bR$-direction: $\overline{\bR}=[-\infty,\infty]$ is a manifold with boundary, and after adding an external collar to $\overline{\bR}$, one obtains the manifold $\widehat{\bR}$. We identify $\cF(X \times \bR)$ and $\cF(X \times \widehat{\bR})$ by means of an orientation-preserving diffeomorphism $h: \bR \to \widehat{\bR}$ with $h ([0,1]) = \overline{\bR}$ and $h((0,1)) = \bR$. Using this identification, we think of elements in $(\Omega \cF)(X)$ as elements of $\cF (X \times \widehat{\bR} )$ which restrict to $z_0$ on $X \times \{\pm \infty \}$.
\end{remark}

A \emph{spectrum of sheaves} is a sequence of pointed sheaves $\cF_n$, $n \geq 0$, and connecting maps $\epsilon_n:\cF_n \to \Omega \cF_{n+1}$. It is called an \emph{$\Omega$-spectrum} if all $\epsilon_n$ are weak equivalences. Taking representing spaces and using the maps \eqref{map-loop-space-sheafworlojjhs}, a spectrum of sheaves induces a spectrum of topological spaces. The \emph{$n$th infinite loop space} of the spectrum $A$ is the homotopy colimit (aka mapping telescope)
\[
 \Omega^{\infty-n} A := \hocolim_r \Omega^{r-n} A_r,
\]
and a spectrum map $T:A \to B$ induces maps $\Omega^{\infty-n} T$ of infinite loop spaces. There is a tautological map $\tau_n:A_n \to \Omega^{\infty-n}A$ which is a weak equivalence if $A$ is an $\Omega$-spectrum. Note that $\Omega^{\infty-n} T \circ \tau_n = \tau_n \circ T_n$.

Our main result involves certain ``maps of spectra'' which are not quite compatible with the connecting maps, but only up to homotopy. To deal with that situation, we introduce the following strictifcation procedure.

\begin{defn}\label{defn:weakspectrummap}
Let $(A_n,\alpha_n)$ and $(B_n,\beta_n)$ be spectra of topological spaces. A \emph{weak spectrum map} is a sequence $T_n:A_n \to B_n$ of pointed maps, such that there are pointed homotopies $\beta_n \circ T_n \sim (\Omega T_{n+1}) \circ   \alpha_n$. A \emph{strictification} of a weak spectrum map $T$ is a spectrum map $\tilde{T}: A \to B$ such that there is a pointed homotopy $\tilde{T}_n \sim T_n :A_n \to B_n $ for each $n$.
\end{defn}

\begin{lem}\label{lem:strictification-of-spectrummaps}
Let $(A_n,\alpha_n)$ and $(B_n,\beta_n)$ be spectra of topological spaces and assume that the adjoint $\alpha^{\ad}_n: \Sigma A_n \to A_{n+1}$ is a cofibration, for each $n \geq 0$. Let $T:A \to B$ be a weak spectrum map. Then $T$ has a strictification $\tilde{T}$.
\end{lem}

\begin{proof}
We construct $\tilde{T}_n$ inductively and set $\tilde{T}_0=T_0$. Assume that $\tilde{T}_k$ is already constructed for $k \leq n$. Then there is a pointed homotopy $(\Omega T_{n+1}) \circ \alpha_n \sim \beta_n \circ T_n \sim \beta_n \circ \tilde{T}_n$ of maps $A_n \to \Omega B_{n+1}$. Taking adjoints yields a pointed homotopy
\[
 T_{n+1} \circ \alpha^{\ad}_n = (\Omega T_{n+1} \circ \alpha_n)^{\ad} \sim (\beta_n \circ \tilde{T}_n)^{\ad} = \beta_n^{\ad} \circ \Sigma \tilde{T}_{n}.
\]
Since $\alpha_n^{\ad}$ is a cofibration, there is $\tilde{T}_{n+1} \sim T_{n+1}: A_{n+1} \to B_{n+1}$ such that $\tilde{T}_{n+1} \circ \alpha^{\ad}_n = \beta_n^{\ad} \circ \tilde{T}_n$.
\end{proof}

\begin{lem}\label{lem:weakhomotopy-infinitelopopmap}
Let $S,T:A \to B$ be two maps of spectra such that $S_m \sim T_m$ for each $m$. Then the maps $\Omega^{\infty-n} S$ and $\Omega^{\infty-n} T$ are weakly homotopic, i.e. they become homotopic when composed with any map $K \to \Omega^{\infty-n} A$ from a finite CW complex. In particular, if $S$ and $T$ are strictifications of the same weak spectrum map, then $\Omega^{\infty-n} S$ and $\Omega^{\infty-n} T$ are weakly homotopic. 
\end{lem}

\begin{proof}
Use that any map from a finite CW complex $K$ to the mapping telescope $\Omega^{\infty-n}A$ factors a finite stage $\Omega^{m-n} A_m$.
\end{proof}

\begin{remark}
The homotopies in Definition \ref{defn:weakspectrummap} are not part of the data. This has the effect that the spectrum map $\tilde{T}$ is \emph{not} uniquely determined up to homotopy. The individual maps $\tilde{T}_n$ are uniquely determined up to homotopy, and the maps $\Omega^{\infty-n} \tilde{T}$ on infinite loop spaces are determined up to weak homotopy, by Lemma \ref{lem:weakhomotopy-infinitelopopmap}. For the rest of the paper, we use the following convention: if $T$ is a weak spectrum map, then the statement that $\tilde{T}$ has a certain property is to be interpreted that \emph{any} strictification $\tilde{T}$ has this property. 
\end{remark}

There are three types of spectra which we like to consider: Thom spectra, $K$-theory spectra, and a spectrum built out of spaces of manifolds. We review the definitions in the next subsections.

\subsection{Vector bundles and Thom spectra}

\begin{defn}
The sheaf $\Vect_d$ of $d$-dimensional vector bundles assigns to $X \in \Mfds$ the set $\Vect_d (X)$ of all smooth real vector bundles $V \subset X \times \bR^{\infty}$ of rank $d$. The subsheaf $\Vect_{d,n} \subset \Vect_d$ assigns to $X$ the set of all $V \in \Vect_d (X)$ with $V \subset X \times \bR^n$. A \emph{vector bundle} of rank $d$ on an arbitrary sheaf $\cF$ is a map of sheaves $\theta: \cF \to \Vect_d$. 
\end{defn}

For example, on the sheaf $\Vect_ {d,n}$ we have the tautological vector bundles $\id:\Vect_{d,n} \to \Vect_{d,n}$ and its orthogonal complement $\bot: \Vect_{d,n} \to \Vect_{n-d,n}$ which sends $V \in \Vect_{d,n}(X)$ to the orthogonal complement bundle $V^\bot \to X$. Of course, the sheaf $\Vect_{d,n}$ is nothing else than the sheaf of smooth maps into the Grassmann manifold $\Gr_{d,n}$. 

\begin{defn}\label{defn:thomsheaf}
Let $\cF$ be a sheaf and let $\theta:\cF \to \Vect_d$ be a vector bundle. The \emph{Thom sheaf} $\cTh (\theta)$ of $\theta$ assigns to $X \in \Mfds$ the set of all triples $(U,z,s)$ where $U \subset X$ is open, $z \in \cF(U)$ and $s$ is a smooth section of the vector bundle $\theta (z) \to U$ which satisfies the following \emph{growth condition}. If $x_n \in U$ is a sequence that converges to $x \in \overline{U}\setminus U$, then $\norm{s(x_n)} \to \infty$. This is a pointed sheaf with basepoint $(\emptyset,*,\emptyset) \in \cTh (\theta)(X)$. 
\end{defn}

To understand the rationale for this definition, consider the example $\cF=\Vect_{d,n}$ and $\theta=\id$ (the $d$-dimensional tautological bundle). The reader should check that in this case $\cTh (\theta)$ is the sheaf of continuous maps $X  \to \Th (V_{d,n})$ of maps to the Thom space of the tautological bundle $V_{d,n} \to \Gr_{d,n}$ which are smooth outside the preimage of the point at infinity.

Let $\theta: \cF \to \Vect_d$ be a $d$-dimensional vector bundle on a sheaf. Let $\cF_n := \theta^{-1}(\Vect_{d,n}) \subset \cF$, let $\theta_n: \cF_n \to \Vect_{d,n}$ be the restriction of $\theta$ and let $\theta^\bot_n : \cF_n \to \Vect_{n-d,n}$ be the orthogonal complement of $\theta_n$, i.e. vector bundle $\cF_n \stackrel{\theta_n}{\to}\Vect_{d,n} \stackrel{\bot}{\to} \Vect_{n-d,n}$. We define 
\[
 \MT \theta_n := \cTh (\theta_n^\bot).
\]
In plain words, $\MT \theta_n (X)$ is the set of all $(U,z,s)$ such that $U \subset X$ is open, $z \in \cF_n (U)$ and $s$ is a smooth section of the vector bundle $\theta (z)^\bot \subset U \times \bR^n$ which satisfies the growth condition. 
The structure map $\eta_n:\MT \theta_n \to \Omega \MT \theta_{n+1}$ sends an element $(U,z,s)$ to $(U \times \bR , \pr_U^* z, s')$, where $s'$ is the section of the bundle $\pr_U^*\theta_n^\bot ( z) \oplus \bR= \theta_{n+1}^{\bot}(\pr_U^* z) \to U \times \bR $ given by $s'(t,x):= (s(x),t)$. Here we use the identification from Remark \ref{rem:chage-coordinates-for-loopsheaf} and view $U \times \bR$ as an open subset of $X \times \widehat{\bR}$.

\begin{defn}
The spectrum $\MT \theta$ just constructed in the \emph{Madsen-Tillmann-Weiss spectrum} of the vector bundle $\theta: \cF\to \Vect_d $. One might write $\MT \theta(d)$ to emphasize the rank of $\theta$.
\end{defn}

\begin{example}\label{ex:dirac-symbol-tangential}
Let us discuss most important (for the purpose of this paper) example of a sheaf with a vector bundle, using the notations introduced in \cite[\S 1.1]{JEIndex1}. Let $\gA$ be a graded Real\footnote{Everything in this paper can easily be ``complexified'', by ignoring the Real structure at every place.} $C^*$-algebra. For a finitely generated projective graded Real Hilbert $\gA$-module $P$ with grading $\eta$, we let $U(P)$ be the group of unitary $\gA$-linear even Real automorphisms of $P$, equipped with the norm topology. This is a Banach Lie group, and hence the notion of a smooth Real graded $P$-bundle on a smooth manifold is well-defined. We define $\cC_\gA$ to be the sheaf which assigns to $X \in \Mfds$ the set of all tuples $(V,Q,\eta,c)$, where 
\begin{enumerate} 
\item $V \to X$ is a real rank $d$ smooth vector subbundle of $X \times \bR^\infty$, equipped with an inner product,
\item $Q \to X$ is a smooth bundle of finitely generated projective Real Hilbert-$\gA$-modules, 
\item $\eta$ is a grading on $Q$ and 
\item $c$ is a $\Cl(V)$-structure on $Q$, in other words, a bundle map $c: V 
\to \End_\gA (Q)$ such that 
\begin{equation}\label{clifford-indeitities}
 c(v)^2 = -\norm{v}^2; \; c(v)^* = - c(v); \; c(v) \eta = -\eta c(v); \; \overline{c(v)} = c(v).
\end{equation}
\end{enumerate}
The map $\theta_\gA:(V,Q,\eta,c) \mapsto V$ is a sheaf map $\theta_\gA: \cC_\gA \to \Vect_d$, and the above construction gives rise to a spectrum $\MT \theta_{\gA}(d)$. We can view $\cC_\gA (X)$ as the set of smooth maps into an infinite-dimensional manifold, as follows. Let $(P,\eta)$ be a graded finitely generated projective Hilbert-$\gA$-module and let $\cS_{d} (P)$ be the set of all Real graded $\Cl^{d,0}$-structures on $P$, in other words, the set of all linear maps $c: \bR^d \to \Lin_{\gA}(P)$ satisfying \eqref{clifford-indeitities} for each $v \in \bR^d$. This is a subset of the normed vector space $\Lin (\bR^d, \Lin_{\gA} (P))$, from which $\cS_d (P)$ inherits its topology. The group $O(d) \times U(P)$ acts on $\cS_d (P)$ via
\[
((g,h) \cdot c)v := h c(gv) h^*. 
\]
Next, we take the disjoint union $\coprod_P \cS_d (P)$, taking one module $P$ from each isomorphism class. 
The Borel construction $EU(P) \times_{U(P)} \coprod_P \cS_d (P)$ can be viewed as the space of all projective finitely generated Hilbert $\gA$-modules equipped with a $\Cl^{d,0}$-structure. It is an $O(d)$-space, and 
\[
\theta_\gA: EO(d) \times_{O(d)} (EU(P) \times_{U(P)} \coprod_P \cS_d (P)) \to BO(d)
\]
is a space model for the map $\theta_\gA$.  
\end{example}

\begin{example}
The construction of the spinor bundle of a spin vector bundle is encoded in a natural map $\MT \Spin (d) \to \MT\theta_{\Cl^{d,0}} (d)$ defined as follows. We let $\cB_{\Spin (d)}$ be the sheaf which assigns to $X \in \Mfds$ the set of all $(V,P,\lambda)$, where $V \in \Vect_d (X)$, $P \to X$ is a smooth $\Spin (d)$-principal bundle and $\lambda: P \times_{\Spin (d)} \bR^d \cong V$ is an isometric isomorphism. This has the homotopy type of $B \Spin (d)$. Let $\MT \Spin (d)$ be the Madsen-Tillmann-Weiss spectrum associated with the forgetful map $\theta:\cB_{\Spin (d)} \to \Vect_d$ defined by $(V,P,\lambda) \mapsto V$.

Recall that $\Spin(d)$ is a subgroup of the multiplicative subgroup of the even part $\Cl^{d,0}_{\ev}$ of $\Cl^{d,0}$. If $(V,P,\lambda) \in \cB_{\Spin (d)}$, then $\spinor_V := P \times_{\Spin (d)} \Cl^{d,0}$ is a bundle of projective finitely generated Hilbert-$\Cl^{d,0}$-modules, with a natural grading $\eta$ and there is a natural map $c: V \to \End (\spinor_V)$ given by Clifford multiplication and $\lambda$. So $(V,P,\lambda) \mapsto (V,\spinor_V,\eta,c)$ defines a map $\cB_{\Spin (d)} \to \cC_{\Cl^{d,0}}$. This induces the map $\MT \Spin (d) \to \MT\theta_{\Cl^{d,0}} (d)$.

More generally, let $G$ be a discrete group. Let $\cB_{\Spin (d) \times G}$ be the sheaf which assigns to $X$ the set of all $(V,P,\lambda,N)$, where $(V,P,\lambda) \in \cB_{\Spin (d)} (X)$ and $N \to X$ is a $G$-Galois cover. The homotopy type of $\cB_{\Spin (d) \times G}$ is $B \Spin (d) \times BG$. A map $\cB_{\Spin (d) \times G} \to \cC_{\Cl^{d,0} \otimes \gC^* (G)}$ is given as follows (here $\gC^*(G)$ can be either the reduced or the maximal group $\cstar$-algebra). It assigns to $(V,P,\lambda,N)$ the element $(V,\spinor_V \otimes \cL_N,\eta \otimes 1, c \otimes 1) \in \cC_{\Cl^{d,0} \otimes \gC^*(G)}$, where $\cL_N \to X$ is the Mishchenko-Fomenko line bundle of $N$. See \cite[\S 1.1]{JEIndex1} for more details. This yields a spectrum map $\MT \Spin (d) \wedge BG_+ \to \MT\theta_{\Cl^{d,0}\otimes \gC^* (G)} (d)$. 
\end{example}

\subsection{K-theory spectra}

In \cite{JEIndex1}, we have defined the model for $K$-theory we are going to use. Let us recall the definition. 

\begin{defn}\cite[Definition 3.4]{JEIndex1}\label{defn-kcycleindex2}
Let $\gA$ be a graded (possibly Real) $C^*$-algebra and $n \geq 0$. A \emph{$K^{n,0}(\gA)$-cycle} on the manifold $X$ is a tuple $(E,\eta,c,D)$, consisting of a continuous field of Hilbert-$\gA$-modules $E$ on $X$, a grading $\eta$ and a $\Cl^{n,0}$-structure $c$ on $E$, and a $\Cl^{n,0}$-antilinear, self-adjoint and odd unbounded Fredholm family $D$ on $E$ (see \cite[Definition 2.32]{JEIndex1}). The cycle $(E,\eta,c,D)$ is \emph{degenerate} if $D$ is invertible. 
\end{defn}

\begin{lem}
The functor $\cK( \gA)_n: \Mfds \to \Set$, which assigns to a test manifold $X$ the set\footnote{As explained in \cite[Remark 3.5]{JEIndex1}, we take a Grothendieck universe and consider all cycles which are contained in this universe.} of all $K^{n,0}(\gA)$-cycles on $X$, is a sheaf.
\end{lem}

\begin{proof}
Let $(U_i)_{i \in I}$ be an open covering of $X$ and let $z_i:=(E_i,\eta_i,c_i,D_i) $ be a compatible family of $K^{n,0}(\gA)$-cycles on the manifolds $U_i$. We construct a $K^{n,0}(\gA)$-cycle $z=(H,\eta,c,D)$ on $X$ as follows. Firstly, there is a unique continuous field of Banach spaces $H$ on $X$ such that $H|_{U_i}= H_i$, by \cite[Proposition 9]{DD}. The fibre $H_x$ of $H$ over $x \in X$ is equal to $(H_i)_x$, where $i \in I$ is so that $x \in U_i$ (it does not matter which $i$ is chosen, since $z_i|_{U_i \cap U_j}= z_j|_{U_i \cap U_j}$). The Hilbert-$\gA$-module structure on $H$, the grading $\eta$ and the $\Cl^{n,0}$-structure $c$ is defined in the unique sensible way. 

Let $(W_i,\Delta_i)$ be the domain of $D_i$ (using the terminology introduced in \cite[\S 2.2]{JEIndex1}). For each $x \in X$, we let $D_x := (D_i)_x$ for suitable $i$. This is an unbounded operator on $H_x$, with a domain $W_x := (W_i)_x$. Now we let $W:= (W_x)_{x \in I}$ and let $\Delta$ be the space of all sections $s $ of $H$ so that $s|_{U_i} \in \Delta_i$ for all $i \in I$. Then $D$ is a closed symmetric operator family with domain $(W,\Delta)$: symmetry is a pointwise condition, and closedness is a local condition. 

The operator family $D$ is Fredholm because $D_i$ is Fredholm, because $D|_{U_i}=D_i$ and by \cite[Lemma 2.18]{JEIndex1}. 
\end{proof}

The basepoint in $\cK(\gA)_n$ is the zero cycle. 
By $\bD(\gA)_n \subset \cK(\gA)_n$, we denote the subsheaf of degenerate cycles. The sheaf $\bD(\gA)_n$ is contractible by \cite[Lemma 3.9]{JEIndex1}. The definition of the group $K^n(X;\gA)$ given in \cite{JEIndex1} can be rewritten as $K^n(X; \gA):= \cK(\gA)_n [X]$. We remark that for compact $X$, this is essentially the unbounded model for the Kasparov group $KK(\Cl^{n,0}, \gC(X,\gA))$.

The Bott map, in the form discussed in \cite[\S 3.3]{JEIndex1}, is a map $\bott:\cK( \gA)_n \to \Omega \cK( \gA)_{n+1}$ of sheaves. Its definition involves the canonical Clifford module which also appears at other places in this paper. 
\begin{defn}\label{defn:cacnonicalclissoefmodules}
 Let $V$ be a euclidean vector space. For $v \in V$, we let $\ins_v: \Lambda^* V^* \to \Lambda^* V^*$ be the insertion operator on the exterior algebra. Let $e(v)$ and $\eps(v)$ be the endomorphisms of $\Lambda^* V^*$ defined by 
\[
e_V (v)=e(v):= \ins_v^* - \ins_v; \;\epsilon_V (v)=\epsilon(v):= \ins_v^* + \ins_v.
\]
Let $\iota=\iota_V$ be the even/odd grading on $\Lambda^* V^*$. Then $V \oplus V^- \to \Lin (\Lambda^* V^*)$, $(v,w) \mapsto e(v) + \eps (w)$ endows $\Lambda^* (V^*)$ with the structure of a graded $\Cl (V \oplus V^-)$-module, denoted $\bS_V$. For $V=\bR^n$, we just write $\bS_n := \bS_V$. In that case, we let $e_i,\eps_i$ be the Clifford action by the standard basis vectors of $\bR^n$.
The construction clearly generalizes to vector bundles. Note that there is a canonical isomorphism
\[
 \bS_V \otimes \bS_W \cong \bS_{V \oplus W}
\]
of Clifford modules (the tensor product of a $\Cl (V)$-module $(E,\iota,c)$ and a $\Cl(W)$-module $(F,\eta,d)$ is the $\Cl (V \oplus W)$-module $(E \otimes F, \iota \otimes \eta, c \otimes 1 + \iota \otimes d)$). The following construction also appears frequently: let $\pi:V \to X$ be a Riemannian vector bundle, $Y$ a space and $f:Y \to V$ a map. By $\eps (f)$, we denote the endomorphism of the vector bundle $(\pi \circ f)^* \bS_V\to Y$ which in the fibre $((\pi \circ f)^* \bS_V)_y = \bS_{V_{\pi(f(y))}}$ is given by $\eps(f(y))$.
\end{defn}

Now we can give the definition of the Bott map. Let $\gx:=(E, \eta,c,D) \in \cK( \gA)_n (X)$ and consider the $K^{n+1,0}(\gA)$-cycle $\gy$ on $\bR \times X $ given by
\[
\gy:=(\pr^*_X E \otimes \bS_1, \eta \otimes \iota, c \otimes 1 + \eta \otimes e, D \otimes 1 + \eta \otimes \eps (\pr_\bR)). 
\]
Explicitly, $\pr^*_X E \otimes \bS_1$ is the continuous field of Hilbert-$\gA$-modules whose fibre over $(t,x)$ is $E_x \otimes \bS_1$, with grading $\eta_x \otimes \iota$. The Clifford action by $v \in \bR^n$ is $c(v) \otimes 1$, and that by $te_{n+1}$ is $\eta \otimes e_1$. The operator over the point $(t,x)$ is $D_x \otimes 1 + \eta_x \otimes t \eps_1$.
The restriction of $\gy$ to $(\bR\setminus [-1,1]) \times X$ is degenerate in the sense of \cite[Definition 3.4]{JEIndex1} and hence $\gy$ can be extended by $0$ along the open embedding $j: \bR \times X \to \widehat{\bR} \times X$, as in \cite[\S 3.1]{JEIndex1}. We put
\[
 \bott (\gx):= j_! \gy 
\]
and obtain a map $\bott:\cK (\gA)_n \to \Omega \cK(\gA)_{n+1}$ of sheaves. It follows from the Bott periodicity theorem in the version \cite[Theorem 3.14]{JEIndex1} that $\bott$ is a weak equivalence of sheaves. 
Thus the collection $(\cK( \gA)_n)_{n \in \bN}$, together with the Bott maps $\cK(\gA)_n \to \Omega\cK(\gA)_{n+1} $ is an $\Omega$-spectrum. The Bott map restricts to a map $\bott: \bD(\gA)_n \to \Omega \bD(\gA)_{n+1}$. Note that $\pi_0 (\cK(\gA)_n) \cong K_{-n}(\gA)$ is the $n$th lower $K$-group of the graded $\cstar$-algebra $\gA$.

\subsection{The Thom homomorphism and the topological index}\label{subsec:thomhom}

In \cite[Definition 3.4]{JEIndex1}, we defined more generally the notion of $K^V (\gA)$-cycles on $X$, where $V \to X$ is a Riemannian vector bundle. Concordance classes of $K^V (\gA)$-cycles on $X$ form an abelian group $K^V (X;\gA)$, which is a twisted version of $K^{\dim (V)} (X;\gA)$. 
\begin{defn}\label{defn:twistedKgroup}
Let $X$ be a manifold and $(V \to X)\in \Vect_d (X)$. A \emph{$K^V(\gA)$-cycle} on $X$ is a tuple $(E,\eta,c,D)$, where $E$ and $\eta$ are as in \eqref{defn-kcycleindex2}, but $c$ is now a $\Cl(V)$-structure on $E$ and $D$ satisfies identities analogous to those spelled out in \eqref{defn-kcycleindex2}. We let $\bK^V (\gA)(X)$ be the set of $K^V (\gA)$-cycles on $X$.

Let $\cF $ be a sheaf and let $\theta:\cF \to \Vect_d$ be a vector bundle. A \emph{$\theta$-twisted $K(\gA)$-cycle} on $\cF$ is an assignment of a $K^{\theta(z)}(\gA)$-cycle $\gx(z)$ on $X$ for each $z \in \cF(X)$. We require naturality of $\gx(z)$, i.e. $f^* \gx(z)= \gx(f^* z)$ for each smooth map $f$.
\end{defn}

\begin{example}\label{ex:twisted-K-cycle-withoutoperator}
Let $\cF$ be the sheaf $\cC_{\gA}$ of Example \ref{ex:dirac-symbol-tangential}, with the forgetful map $\theta: \cC_\gA \to \Vect_d$. Let $z:=(V,Q,\eta,c) \in \cC_{\gA}(X)$ (recall that $\theta(z)=V$). We define a $K^V (\gA)$-cycle
\[
 \gx (z):= (Q,\eta,c,0).
\]
Note that $0$ is a Fredholm family because $Q$ is a bundle of \emph{finitely generated} projective modules. 
\end{example}

Next, we introduce the Thom isomorphism (we do not need to know that it is an isomorphism). To that end, let $\theta: \cF \to \Vect_{d,n}$ be a vector bundle with complement $\theta^\bot: \cF \to \Vect_{n-d,d}$ and let $\gx$ be a $\theta$-twisted $K(\gA)$-cycle on $\cF$. We wish to construct a sheaf map
\begin{equation}\label{thom-iso-as-sheaf-map}
\thom(\gx): \cTh (\theta^\bot) \to \cK(\gA)_n
\end{equation}
out of these data. Let $(U,z,s) \in \cTh (\theta^\bot)(X)$. Recall that $U \subset X$ is open, with inclusion map $j$, $z \in \cF(U)$, that $\theta(z)\subset U \times \bR^n$ is a vector bundle with complement $\theta(z)^\bot$. Finally, $s$ is a section of $\pi^\bot:\theta(z)^\bot \to U$ with the growth condition of Definition \ref{defn:thomsheaf}. 
The $\theta(z)$-twisted $K(\gA)$-cycle $\gx(z)$ can be written as $(E,\eta,c,D)$. We define
\begin{equation}\label{thom-iso-as-sheaf-map2}
 \thom(\gx)(U,z,s):= j_! (E\otimes \bS_{\theta^\bot(z)},\eta \otimes \iota_{\theta^\bot(z)}, c \otimes e_{\theta^\bot(z)}, D \otimes 1 + \eta \otimes \eps(s))
\end{equation}
using the extension-by-zero map $j_!$. The tensor product is the tensor product of a continuous field with a finite-dimensional vector bundle (and hence unproblematic). Since $D$ is odd,
\[
 (D \otimes 1 + \eta \otimes \eps(s))^2 = D^2 \otimes 1 + 1 \otimes \eps(s)^2 \geq \norm{s}^2
\]
and by the growth condition on $s$, extension by $0$ is indeed well-defined. Note that the Bott map can be viewed as a special case of the Thom homomorphism.

Now consider slightly more generally a sheaf with a vector bundle $\theta:\cF\to \Vect_d$ and a $\theta$-twisted $K(\gA)$-cycle $\gx$ on $\cF$. It restricts to a $\theta_n$-twisted $K (\gA)$-cycle $\gx_n$ on $\cF_n$. The above construction yields maps 
\[
 \thom(\gx_n):  \MT \theta_n=\cTh (\theta_n^\bot) \to \cK(\gA)_n
\]
of sheaves.

\begin{lem}\label{lem:thom-class-is-spectrum-map}
The sheaf maps $\thom(\gx)_n$ assemble to a weak spectrum map $\thom (\gx):\MT \theta \to \cK( \gA)$, in the sense that the diagram
\[
\xymatrix{
\MT \theta_n \ar[r]^-{\eta_n} \ar[d]^{\thom(\gx)_{n}} & \Omega \MT \theta_{n+1}  \ar[d]^{\Omega \thom(\gx)_{n+1}}\\
\cK(\gA)_n \ar[r]^-{\bott} & \Omega \cK(\gA)_{n+1}
}
\]
commutes up to a natural concordance. If the cycles $\gx$ and $\gy$ are naturally concordant, then $\thom (\gx_n)$ and $\thom (\gy_n)$ are homotopic. 
\end{lem}

\begin{proof}
This is by a straightforward unwinding of the definitions involved. One uses the natural isomorphism $\bS_V \otimes \bS_W \cong \bS_{V \oplus W}$ and that an isomorphism of $K^{n,0}(\gA)$-cycles yields a concordance, in a natural way, by \cite[Lemma 3.6]{JEIndex1}.
\end{proof}

\begin{defn}\label{defn:topological-indexnasn}
Let $\gA$ be a graded Real $\cstar$-algebra and let $\cC_\gA \to \Vect_d$ be the sheaf with vector bundle defined in Example \ref{ex:dirac-symbol-tangential}. Let $\gx$ be the $\theta$-twisted $K(\gA)$-cycle on $\cC_\gA$ constructed in Example \ref{ex:twisted-K-cycle-withoutoperator}. The weak spectrum map 
\[
 \topind:= \thom (\gx): \MT \theta_\gA(d) \to \cK (\gA)
\]
is the \emph{topological index}. 
\end{defn}

\subsection{Spaces of manifolds}

We now discuss the spectrum $\GRW \theta$ of spaces of manifolds, which was introduced by Galatius and Randal-Williams in \cite{GRW10} (under a different name). 

\begin{defn}
Let $\pi: M \to X$ be a submersion of manifolds with $d$-dimensional fibres. The \emph{vertical tangent bundle} $T_v \pi=T_v M \to M$ is the rank $d$ vector bundle $\ker (d \pi)$. 
A map $f: M \to \bR^n$ is \emph{fibrewise proper} if $(\pi,f):M \to X \times \bR^n$ is proper (note that the restriction of $f$ to $M_x := \pi^{-1}(x)$ is then a proper map to $\bR^n$).
\end{defn}

\begin{defn}
Let $\theta: \cF \to \Vect_d$ be a vector bundle on a sheaf $\cF$. Let $\pi:M \to X$ be a submersion with $d$-dimensional fibres. A \emph{$\theta$-structure} on $M$ is an element $\ell \in \cF(M)$ such that $\theta (\ell)=T_v M$.
\end{defn}

In order to have a well-behaved notion, we need to assume that the map $\theta$ of sheaves has the \emph{concordance lifting property}, which we shall assume henceforth. For the definition of this term, see \cite[Definition 4.5]{MW}; this is a version of the homotopy lifting property in the context of sheaves. Our main example, the map $\theta_\gA: \cC_\gA \to \Vect_d$ from Example \ref{ex:dirac-symbol-tangential}, has the concordance lifting property. 

\begin{defn}
Let $k \geq n$. For a test manifold $X$, let $\cD_{\theta,n}^k (X)$ be the set of all pairs $(M,\ell)$, where 
\begin{enumerate}
\item $M \subset X \times \bR^k$ is a submanifold which is closed as a subspace,
\item the projection $\pi=\pr_X:M \to X$ to the first factor is a submersion with $d$-dimensional fibres,
\item $\ell$ is a $\theta$-structure on $M$,
\item the projection map $f=\pr_{\bR^n}: M \to \bR^n$ onto the first $n$ coordinates is fibrewise proper.
\end{enumerate}
This defines a sheaf $\cD_{\theta,n}^k$.
\end{defn}

There are obvious inclusion maps $j: \cD_{\theta,n}^k \subset \cD_{\theta,n}^{k+1}$, and we define
\[
 \GRW\theta_n = \GRW\theta(d)_n:= \colim_k \cD_{n,\theta}^k.
\]
We remark that the colimit is to be understood in the category $\Sheaf$; the colimit of a sequence $\cF_0 \to \cF_1 \to \ldots$ is the sheafification of the presheaf $X \mapsto \colim_n (\cF_n(X))$. Let $(M, \ell)\in \GRW\theta_n (X)$. For each $x \in X$, the fibre $\pi^{-1}(x)$ is a $d$-dimensional submanifold of $\bR^\infty$, equipped with a $\theta$-structure, and the map $f:\pi^{-1}(x)\to\bR^n$ is proper. If $n \geq 1$, the diffeomorphism type of $\pi^{-1}(x)$ can change drastically with $x$, but if $n=0$, the set $\GRW\theta_0 (X)$ consists of all bundles of closed manifolds on $X$ (embedded into $\bR^\infty$), equipped with a $\theta$-structure, by Ehresmann's fibration lemma. We think of $\GRW\theta_n$ as the moduli space of $\theta$-manifolds which are ``noncompact in $n$ directions'' or ``controlled over $\bR^n$''. 

\begin{defn}\label{defn:scanning-map}
For $n<k$, the \emph{scanning map}
\begin{equation}\label{scanning-map}
\sigma: \cD_{\theta,n}^k \to \Omega \cD_{\theta,n+1}^{k}
\end{equation}
is defined as follows. Let $(M,\ell) \in \cD_{\theta,n}^k (X)$. Let $\sigma (M):= \{ (t,x,z) \in  \bR \times X  \times \bR^k \vert (x,z - t e_{n+1}) \in M\}$. This is a submanifold of $\bR \times X \times \bR^k$ and closed in $\widehat{\bR} \times X \times   \bR^k$. The projection onto $\widehat{\bR} \times X $ is a submersion with $d$-dimensional fibres (which are either diffeomorphic to $M$ or empty). The map $h:\bR \times M \to \sigma (M)$, $(t,x,z) \mapsto (t,x,z+te_{n+1})$, is a diffeomorphism over $\bR \times X$. This identifies the vertical tangent bundle of $\sigma (M)$ with the pullback of $T_v M$ along the projection $\bR \times M \to M$, and $\sigma (\ell)$ is the pulled back $\theta$-structure. 
\end{defn}
It is clear from the definitions that the diagram
\[
\xymatrix{
\cD_n^k \ar[r]^{j} \ar[d]^{\sigma} &\cD_n^{k+1} \ar[d]^{\sigma} \\
\Omega \cD_{n+1}^k \ar[r]^j & \Omega \cD_{n+1}^{k+1}
}
\]
commutes. Therefore, the scanning maps $\sigma$ induce a map 
\[
\scan: \GRW \theta_n \to \Omega \GRW \theta_{n+1}
\]
which turns $\GRW\theta$ into a spectrum. 
\begin{thm}[Galatius, Randal-Williams \cite{GRW10}]\label{grwtheorem2}
The spectrum $\GRW \theta$ is a weak $\Omega$-spectrum in the sense that for all $n \geq 1$, the maps $\GRW\theta_n \to \Omega \GRW\theta_{n+1}$ are weak equivalences. 
\end{thm}
In \S \ref{subsec:grwproofs} below, we show how to derive Theorem \ref{grwtheorem2} from the results actually stated in \cite{GRW10}. 

\begin{defn}\label{linear-mfds-to-all-mfds}
A map
\[
\lambda_n:\MT \theta_n \to \cD_{\theta,n}^n
\]
of sheaves is defined by the following procedure. Let $(U,z,s) \in \MT \theta_n (X)$, i.e. $U \subset X$ is open, $z \in \cF(U)$, $\theta(z) \subset U \times \bR^n$ is a rank $d$ vector bundle with bundle projection $\pi$ and $s$ is a smooth section of the complement $\theta(z)^\bot$, subject to the growth condition. Define 
\[
 f: \theta(z) \to \bR^n; \; f (x,v):= v+s(x).
\]
The map $(\pi,f)$ is a proper embedding $\theta(z) \to X \times \bR^n$: it is clearly injective, and easily seen to be an immersion. To verify that it is proper, let $(x_n,v_n)\in \theta(z)$ be a sequence such that $(x_n, w_n+s(x_n))$ converges to $(x,z)\in X \times \bR^n$. Since $v_n \bot s(x_n)$, we have $\norm{v_n+s(x_n)}^2 = \norm{v_n}^2 + \norm{s(x_n)}^2$. Hence $\norm{s(x_n)}$ is bounded, and this implies that $x \in U$ and $s(x_n)\to s(x)$, by the growth condition. Then $v_n \to z - s(x)$, and $(x,z-s(x))\in \theta(z)$.
So $M:= (\pi,f)(\theta(z))$ is an element of $\cD_n^n(X)$. The vertical tangent bundle $T_v M:= \ker d \pi $ is canonically identified with $\pi^* \theta(z)$, and in particular, it is equipped with a canonical $\theta$-structure. 
\end{defn}
It follows quickly from the definitions that the diagram
\begin{equation}\label{structuremap-vs-scanning}
\xymatrix{
\MT \theta_n \ar[d]^{\eta_n} \ar[r]^{\lambda_n} & \cD_n^n \ar[r]^{j} & \cD_n^{n+1} \ar[d]^{\sigma}\\
\Omega \MT \theta_{n+1} \ar[rr]^{\Omega \lambda_{n+1}} & & \Omega \cD_{n+1}^{n+1}
}
\end{equation}
commutes. Hence the maps 
\[
\Lambda_n : \MT \theta_n \stackrel{\lambda_n}{\to} \cD_{n}^n \to \GRW\theta_n
\]
define a spectrum map 
\[
\Lambda: \MT \theta \to \GRW\theta.
\]
\begin{thm}[Galatius, Randal-Williams \cite{GRW10}]\label{grwtheorem1}
The map $\Lambda_n$ is $ (2n-2d-1) $-connected for each $n \geq 1$. In particular, $\Lambda$ is a stable weak equivalence of spectra. 
\end{thm}

Again, this is not stated as such in \cite{GRW10}. The derivation of Theorem \ref{grwtheorem1} from \cite{GRW10} uses ideas that are unimportant for the rest of this paper, and is therefore deferred to \S \ref{subsec:grwproofs}.

\begin{remark}\label{rem:alternativeviewpoint-scanning}
It is useful to change the perspective on elements of $\GRW \theta_n (X)$ slightly. Instead of remembering that $M \subset X \times \bR^\infty$ and that the projection map to $x$ is a submersion and that to $\bR^n$ is fibrewise proper, one can explicitly record them as $\pi$ and $f$ in the data. Hence we may think about elements of $\GRW\theta_n (X)$ as tuples $(M,\pi,f,\ell)$, $\pi:M \to X$ a submersion, $\ell$ a $\theta$-structure, and $f:M \to \bR^n$ a fibrewise proper map.

In this picture, the scanning map has an easier description: it maps $(M,\pi,f,\ell)$ to $(\bR \times M ,\pi',f',\ell')$, where $\pi'= \id \times \pi: \bR \times M \to \widehat{\bR} \times X$, $\ell'$ is the $\theta$-structure induced by $\theta$ via the canonical isomorphism $T_v \pi' \cong \pr_M^* T_v \pi$. Finally, $f'(t,x):= (f(x),t)$. 

This viewpoint simplifies the description of $\Lambda_n$ as well. It maps $(U,z,s) \in \MT \theta_n$ to $(\theta(z),\pi,f,\ell)$, where $\pi: \theta(z) \to U$ is the bundle projection, $f: \theta(z)\to \bR^n$ is the map from Definition \ref{linear-mfds-to-all-mfds} and $\ell$ is the canonical $\theta$-structure.
\end{remark}

\begin{remark}\label{rem:pontrjaginthom}
 The reader of \cite{GMTW} might have expected maps $\GRW\theta_n \to \Omega^{\infty-n} \MT \theta$ coming from a parametrized Pontrjagin-Thom construction to play an important role. These can be abstractly constructed, as follows (at least after taking representing spaces of the sheaves involved). The spectra $\GRW \theta$ and $\MT \theta$ of sheaves induce spectra $|\GRW \theta|$ and $|\MT \theta|$ of spaces, as explained in \S \ref{subsec:homotopytheory}. The map $ \Omega^{\infty-n}|\Lambda|: \Omega^{\infty-n} |\MT \theta| \to \Omega^{\infty-n}|\GRW\theta|$ is a weak homotopy equivalence by Theorem \ref{grwtheorem1}. We let $p_n: \Omega^{\infty-n}|\GRW\theta| \to \Omega^{\infty-n}|\MT\theta|$ be a homotopy inverse and write $\PT_n := p_n \circ \tau_n: |\GRW  \theta_n| \to \Omega^{\infty-n} |\MT \theta|$. For $n \geq 1$, this is a weak equivalence, by Theorem \ref{grwtheorem2}. One may construct the map $\PT_n$ geometrically by means of a Pontrjagin-Thom construction, similar to \cite[\S 3.1]{GMTW}, but 
that is not important for us. 
\end{remark}

\subsection{Proof of Theorems \ref{grwtheorem2} and \ref{grwtheorem1}}\label{subsec:grwproofs}

\begin{proof}[Proof of Theorem \ref{grwtheorem2} from \cite{GRW10}]
In \cite[\S 2]{GRW10}, a topology on the set $\cD_{\theta,n}^k(*)$ is defined, and the resulting space is denoted $\Psi_{\theta} (n,k)$ in loc.cit. An element $(M, \ell)\in \cD_{\theta,n}^k(X)$ defines a continuous map $X \to \Psi_{\theta} (n,k)$, $x \mapsto (\pi^{-1}(x),\ell|_{\pi^{-1}(x)})$ (it is even a smooth map in the sense of Definition 2.15 loc.cit.). Therefore, we obtain a map $\cD_{\theta,n}^k \to \sh (\Psi_{\theta} (n,k))$. Using \cite[Lemma 2.17]{GRW10}, one can show that this is a weak equivalence. 

There is an unnamed map ((3-10) in \cite{GRW10}) $\Psi_{\theta}(n,k) \to \Omega \Psi_{\theta}(k+1,k)$, which corresponds to the map $\sigma$; and Theorem 3.13 of \cite{GRW10} says that this map is a weak equivalence if $n \geq 1$. Hence so is $\sigma$. Passage to the colimit $k \to \infty$ finishes the proof of Theorem \ref{grwtheorem2}.
\end{proof}

To derive Theorem \ref{grwtheorem1} from \cite{GRW10}, we need an input from classical homotopy theory.

\begin{lem}\label{connectivity-thom-spaces}
Let $f: X \to Y$ be an $r$-connected map between spaces, let $W \to Y$, $V \to X$ be vector bundles, of rank $s+1$ and $s$, respectively, and let $V \oplus \bR \cong f^* W$ be an isomorphism. We get maps of Thom spaces
\[
\Th (V) \to \Omega \Th (V\oplus \bR) \to \Omega \Th (W).
\]
The composition of those maps is $\min \{2s-1, r+s \}$-connected. 
\end{lem}

\begin{proof}
Since $\Th (V)$ is $(s-1)$-connected, the Freudenthal suspension theorem implies that the first of those maps is $(2s-1)$-connected. By the Thom isomorphism with twisted coefficients and the Hurewicz theorem, the second map is $(r+s)$-connected. \end{proof}

\begin{lem}\label{cor:connectivity-spectrum-map}
The map $\MT \theta_n \to \Omega \MT\theta_{n+1}$ is $(2n-2d-1)$-connected.
\end{lem}

\begin{proof}
Let $\theta: \cF \to \Vect_d$ be the underlying map of sheaves with the concordance lifting property. The diagram
\[
 \xymatrix{
 \cF_n \ar[r] \ar[d]^{\theta} & \cF_{n+1} \ar[d]^{\theta}\\
 \Vect_{d,n} \ar[r] & \Vect_{d,n+1}
 }
\]
induces a homotopy cartesian diagram after taking representing spaces, since $\theta$ has the concordance lifting property and by \cite[Proposition A.6]{MW}. The bottom map is homotopy equivalent to the inclusion map $\Gr_{d,n} \to \Gr_{d,n+1}$ of Grassmann manifolds, which is $(n-d)$-connected. Therefore $|\cF_n|\to |\cF_{n+1}|$ is $(n-d)$-connected as well. The map $|\MT \theta_n| \to |\Omega \MT \theta_{n+1}|$ is homotopy equivalent to a map of Thom spaces over $|\cF_{n} |\to |\cF_{n+1}|$. Hence by Lemma \ref{connectivity-thom-spaces}, it is $\min \{2 (n-d)-1,(n-d+1) + (n-d)\}  = ( 2n-2d-1)$-connected.
\end{proof}

\begin{proof}[Proof of Theorem \ref{grwtheorem1}]
For a map $f: X \to Y$, we write $\conn (f)$ for the largest $r$ such that $f$ is $r$-connected. Assume that $n \geq 1$. The map $\Lambda_n$ was defined as the composition
\[
\Lambda_n : \MT \theta_n \stackrel{\lambda_n}{\to} \cD_{\theta,n}^{n} \to  \cD_{\theta,n}^{n+1} \to  \cD_{\theta,n}^{n+2} \to \ldots \to\GRW\theta_n.
\]
The map $\lambda_n$ is a weak equivalence by \cite[Theorem 3.22]{GRW10} (or rather a sheaf version of that result). Therefore
\[
\conn(\Lambda_n) \geq \min \{\conn (j:\cD_{\theta,n}^k \to \cD_{\theta,n}^{k+1}) \vert k \geq n\}.
\]
For $k \geq n$, the diagram
\[
\xymatrix{
\cD_{\theta,n}^k \ar[d]^{\sigma} \ar[r]^{j} & \cD_{\theta,n}^{k+1} \ar[d]^{\sigma}  \\
\Omega^{k-n} \cD_{\theta,k}^k  \ar[r]^{ \Omega^{k-n}j} & \Omega^{k-n} \cD_{\theta,k}^{k+1} \ar[d]^{\Omega^{k-n}\sigma} \\
& \Omega^{k+1-n} \cD_{\theta,k+1}^{k+1}
}
\]
commutes, with the iterated scanning maps as vertical maps. By \cite[Theorem 3.13]{GRW10}, all vertical maps are weak equivalences. But the composition $\cD_{\theta,k}^k \stackrel{j}{\to} \cD_k^{\theta,k+1} \stackrel{\sigma}{\to} \Omega \cD_{\theta,k+1}^{k+1}$ is homotopy equivalent to the structure map $\MT\theta_k \to \Omega \MT\theta_{k+1}$, by \eqref{structuremap-vs-scanning} and Theorem \cite[Theorem 3.13]{GRW10}. By Lemma \ref{cor:connectivity-spectrum-map}, it follows that $\cD_{\theta,n}^k \to \cD_{\theta,n}^{k+1}$ is $(2k-2d-1)- (k-n) =(k-2d-1+n)$-connected. Therefore
\[
\conn (\Lambda_n) \geq \min \{ k-2d-1+n \vert k \geq n\} =  2(n-d)-1. \qedhere
\]
\end{proof}

\section{The spectrum of manifolds equipped with Dirac operators and the index map}\label{sec:construction-index-amp}

\subsection{Spaces of manifolds equipped with Dirac operators}

Throughout this section, we fix a dimension $d$ (the dimension of the manifolds we are interested in) and a graded and possibly Real $C^*$-algebra $\gA$. 
The map $\theta_{\gA} : \cC_{\gA} \to \Vect_d$ defined in Example \ref{ex:dirac-symbol-tangential} has the concordance lifting property and yields spectra $\MT \theta_{\gA}(d)$ and $\GRW \theta_{\gA}(d)$. To ease notation, we shall write $\MT \gA$ and $\GRW \gA$ for those spectra. An element of $\GRW\gA_n(X)$ is a tuple $(M,\pi,f,E,\eta,c)$, where $M$ is a manifold equipped with a submersion $\pi:M \to X$ with $d$-dimensional fibres, $f: M \to \bR^n$ is a fibrewise proper map, $(E,\eta)$ is a bundle of graded, finitely generated projective Hilbert-$\gA$-modules on $M$ with a $\Cl (T_v M)$-structure $c$ (note that a $\theta_\gA$-structure contains a smooth metric on the fibres of $\pi$). Also, $M$ is a subset of $X \times \bR^\infty$, and $\pi$ and $f$ are the respective projection maps.

Recall that a \emph{Dirac operator} on such a bundle $E$ equipped with $\eta$ and $c$ is a fibrewise, $\gA$-linear, formally self-adjoint odd differential operator of order $1$ so that for each function $h : M \to \gC$, $\smb_D (dh):=  i[D,h]=ic(dh)$ ($\smb_D (\_)$ is the \emph{symbol} of $D$).
We want to define a version $\GRW \gA^{op}$ of the spectrum $\GRW \gA$ which has Dirac operators on $E$ as an additional piece of datum. 

For example, one could try to use the sheaf that takes $X$ to the set of all $(M,\pi,f,E,\eta,c,D)$ with $(M,\pi,f,E,\eta,c) \in \GRW \gA_n (X)$ and $D$ is a Dirac operator on $E$. We would like to define a spectrum map $\GRW \gA^{op} \to \bK \gA$ that takes the index of the operator $D$ in an appropriate sense.

However, as it stands, the operators $D$ are not suited for analytical arguments. The problem is that the pair $(M,D)$ is not necessarily complete in the sense of \cite[Definition 1.13]{JEIndex1}, so that $D$ is not necessarily self-adjoint. In the absence of self-adjointness, there is not much operator theory available for the operators $D$. 
One could try to allow only those operators $D$ such that $(M,D)$ is complete in the definition of $\GRW \gA^{op}$, but it is more convenient to include \emph{more data} into the definition instead.

\begin{defn}\label{defn:space-of-mfds-with-diracs-better}
For a test manifold $X$, $\GRW \gA^{op}_n(X)$ is the set of all tuples $ (M, \pi,f,E,\eta,c,D,g)$ where 
\begin{enumerate}
\item $(M, \pi,f,E,\eta,c) \in \GRW \gA_n (X)$,
\item $D$ is a Dirac operator on $E$ and
\item $g: M \to (0,\infty)$ is a \emph{moderating function}, i.e. a smooth function with the following property: writing $f_j: M \to \bR$ for the $j$th component of $f$, we require that the commutator
\[
[gDg, f_j]
\]
is locally (in $X$) bounded, for each $j=1, \ldots, n$.
\end{enumerate}
\end{defn}

\begin{lem}\label{lem:forgetting-diff-data}
The forgetful map $\xi:\GRW \gA^{op}_n \to \GRW \gA_n$ is a weak homotopy equivalence.
\end{lem}

\begin{proof}
There is a familiar lifting criterion for a map of spaces to be a weak equivalence. In the context of sheaves, this is stated as \cite[Proposition 2.18]{MW}. What we have to prove is the following statement. 
Let $X$ be a test manifold and let $Y \subset X$ a closed subset. Let $ (M, \pi,f,E,\eta,c) \in  \GRW \gA_n(X)$. Assume that there is a neighborhood $U$ of $Y$ and a lift $(M|_U, \pi|_U,f|_U,E|_U,\eta|_U,c|_U,D_U,g_U) \in \GRW\gA^{op} (U)$ defined over $U$. 
Then we can find a possibly smaller neighborhood $U_0 \subset U$ of $A$ and a lift $(M, \pi,f,E,\eta,c,D,g)$ over $X$ which coincides with the given one on $U_0$. The data $(M,\pi,f,E,\eta,c)$ are untouched and will be suppressed in the notation. 

That we can define $D$ is a consequence of the well-known fact that differential operators with prescribed symbols can always be constructed (and there is no problem making them odd, Real self-adjoint if that is required). More precisely, we can find some Dirac operator $D'$ on $E$, defined over all of $X$. Choose a smooth function $\mu: X \to [0,1]$ which is $1$ near $Y$ and has support in $U$ and form $D := \sqrt{\mu} D_U \sqrt{\mu}+ \sqrt{1-\mu} D'\sqrt{1-\mu}$. 
To show that $g_U$ can be extended, let first $h : M \to (0,\infty)$ be any smooth function. Then 
\[
[h^{1/2} D h^{1/2}, f_j]  =  h^{1/2} [D , f_j ] h^{1/2} = h [D,f_j],
\]
the last equation holds because $[D,x_j]$ is of order $0$. This proves that the space of all $h: M \to (0,\infty)$ such that the commutators $[h^{1/2} D h^{1/2}, f_j]$ are all locally bounded is convex and nonempty. Pick one such function $h$. The desired extension is then $g:=\sqrt{\mu g_U^2 + (1-\mu) h}$.
\end{proof}

The notation for elements in $\GRW \gA^{op}_n$ is cumbersome. We therefore often shorten notation by only writing those parts of the datum which are relevant for the argument in question. 

\begin{lem}\label{lem:grw-op-spectrudsf}
There is a scanning map $\scan=\scan^{op}:\GRW\gA^{op}_n \to \Omega \GRW\gA^{op}_{n+1}$ such that the diagram
\[
\xymatrix{
\GRW\gA^{op}_n \ar[r]^{\scan^{op}} \ar[d]^{\xi}& \Omega \GRW\gA^{op}_{n+1}\ar[d]^{\xi}\\
\GRW\gA_n \ar[r]^-{\scan} & \Omega \GRW\gA_{n+1}
}
\]
commutes. In particular, $\scan^{op}$ is a weak equivalence if $n \geq 1$.
\end{lem}

\begin{proof}
Let $(M, \pi,f,E,D,g)\in \GRW\gA^{op} (X)$. Using the description of the scanning map given in Remark \ref{rem:alternativeviewpoint-scanning}, the composition $\scan \circ \xi$ sends this element to $(\bR \times M, \id \times \pi, f',\pr_M^* E)$, where $ f'_j = f_j \circ \pr_M$ for $j \leq n$ and $f'_{n+1}= \pr_\bR$. We let $D'$ be the pulled back operator on $\pr_M^* E$ and define $g':= g \circ \pr_M$. Then $g'$ is a moderating function. To see this, we have to show that $[g'D'g',f_j']$ is bounded (locally in $\widehat{\bR} \times X$). For $j \leq n$, this follows from the assumption that $g$ is a moderating function, and for $j=n+1$, one observes that $[g'D'g',f_{n+1}']=0$. Now we define the scanning map $\scan^{op}$ by 
\[
(M, \pi,f,E,D,g)\mapsto (\bR \times M, \id \times \pi, f',\pr_M^* E,D',g').
\]
The last sentence follows from Theorem \ref{grwtheorem2} and Lemma \ref{lem:forgetting-diff-data}.
\end{proof}

\begin{lem}\label{lem:from-linear-mfds-to-dirac}
There is a map of spectra $\Lambda^{op}:\MT \gA \to \GRW \gA^{op}$ such that $\xi \circ \Lambda^{op} = \Lambda$. In particular, $\Lambda^{op}$ is a stable weak equivalence of spectra.
\end{lem}

\begin{proof}
Let $(U,W,E,\eta,c,s)\in \MT \gA_n (X)$, which under $\Lambda_n$ maps to $(W,\pi,f,\pi^* E,\eta,c)$, where $f$ is the map defined in (\ref{linear-mfds-to-all-mfds}).
We define the Dirac operator $D_E: \Gamma_{cv} (W;\pi^* E)\to \Gamma_{cv} (W;\pi^* E)$ as follows. First we fix $x \in U$ and define $D_{E,x} $ on $C^{\infty}_c (W_x; E_x)$ using an orthonormal basis $(w_1, \ldots,w_d)$ of $W_x$ by the formula
\[
D_{E,x} := \sum_{j=1}^d c(w_i) \partial_{w_i}.
\]
These operators fit together to a family $D_E$ of elliptic operators on $\pi:W\to U$. The fibrewise differential of the function $f_j: W \to \bR$, $(x,w) \mapsto \scpr{w,e_j}+ \scpr{s(x),e_j}$ is the same as the fibrewise differential of the coordinate function $l_j:w \mapsto \scpr{w,e_j}$. It follows that $[D_E,f_j]= - i \smb_{D_E} (l_j) = c (l_j)$, which is clearly bounded. Therefore, $g=1$ is a moderating function. Define
\[
\Lambda_n^{op}(U,W,E,\eta,c,s):= (W,\pi,f,\pi^* E,\eta,c,D_E,1) \in \GRW\gA_n^{op}(X).
\]
It is straightforward to check that the collection $(\Lambda_n^{op})_n$ is a map of spectra, and it is clear that $\xi \circ \Lambda^{op}=\Lambda$. The last sentence follows from Theorem \ref{grwtheorem1} and Lemma \ref{lem:forgetting-diff-data}.
\end{proof}

\subsection{Construction of the analytic index map}\label{subsec:construction-indexmap}

We are now ready to define the analytic index map
\[
\ind_n : \GRW \gA_n^{op} \to \bK( \gA)_n.
\]
For the rest of this subsection, fix a test manifold $X$ and $ (M, \pi,f,E,\eta,c,D,g) \in \GRW \gA^{op}_n (X)$. To assign to these data an element in $\bK(\gA)_n(X)$, we use the analytical results from \cite{JEIndex1}. 

Write $M_x := \pi^{-1}(x)$, $E_x:= E|_{M_x}$ and $D_x$ for the restriction of $D$ to $M_x$. In \cite[Example 2.12]{JEIndex1}, we constructed a continuous field $L_X^2 (M;E)$ of Hilbert-$\gA$-modules. Its fibre over $x \in X$ is the Hilbert-$\gA$-module $L^2 (M_x;E_x)$, the completion of the space $\Gamma_c (M_x;V_x)$ of compactly supported smooth sections with the $\gA$-valued inner product induced by the scalar product on $E$ and the volume measure on $M_x$ (recall that by definition, $M_x$ has a Riemannian metric). The space $\Gamma_{cv} (M;E)$ of vertically compactly supported sections of $E$ is a total subspace of $L^2_X(M;E)$. The weighted Dirac operator $gDg$ is a differential operator family of order $1$, and it is a densely defined symmetric unbounded operator family with initial domain $\Gamma_{cv} (M;E) \subset L^2_X (M;E)$. We first prove that the closure of $gDg$ is a self-adjoint family in the sense of \cite[Definition 2.25]{JEIndex1}.

\begin{lem}\label{lem:selfadjointness-weighted-dirac}
The closure of the weighted Dirac operator $gDg$ is a self-adjoint operator family on $L^2_X (M;E)$. 
\end{lem}

\begin{proof}
The differential operator $gDg$ is formally self-adjoint, because $g$ is real-valued. We want to apply \cite[Theorem 1.14]{JEIndex1}, and for that to work, we need a coercive function $h: M \to \bR$ (see \cite[Definition 1.12]{JEIndex1}) such that $[D,h]$ is locally bounded (in $X$). Define
\[
h: M \to \bR; \; h(y):=(1 + \norm{f(y)}^2)^{1/2}. 
\]
It is clear that $h$ is coercive, i.e. fibrewise proper and bounded from below.
We claim that $[gDg,h]$ is (locally in $X$) bounded. But $D$ has order $1$, whence
\[
[gDg,h] = -i\smb_{gDg} (dh) = - i \smb_{gDg} (\sum_{j=1}^{n} \frac{f_j}{(1+\norm{f}^2)^{1/2}} df_j) = 
\]
\[
=- i \sum_{j=1}^{n} \frac{f_j}{(1+\norm{f}^2)^{1/2}}\smb_{gDg} (  df_j) =   \sum_{j=1}^{n} \frac{f_j}{(1+\norm{f}^2)^{1/2}}[gDg, f_j] .
\]
Since $[gDg,f_j]$ is locally bounded (in $X$), it follows that $[gDg,h]$ is locally bounded (in $X$). Hence by \cite[Theorem 1.14]{JEIndex1}, the restriction of $gDg$ to each fibre $\pi^{-1}(x)$ is essentially self-adjoint. By \cite[Definition 2.25]{JEIndex1}, the proof is complete. See also \cite[Example 2.28]{JEIndex1} for more details on this last step.
\end{proof}

Usually, $gDg$ is not Fredholm unless $n=0$. To make up a Fredholm operator, we take a suitable tensor product with the canonical $\Cl^{n,n}$-module $\bS_{n}$ from Definition \ref{defn:cacnonicalclissoefmodules}. 
The (graded) tensor product bundle $E\otimes \bS_{n} \to M$ has the grading $\eta \otimes \iota$ and the $\Cl (T_v M \oplus \bR^{n,n})$-structure $(v,w,u) \mapsto c(v) \otimes 1 + \eta \otimes (e(w)+\eps(u))$. The map $f:M \to \bR^n$ gives the order $0$ operator $\eps (f): C^{\infty} (M; \bS_{n})\to C^{\infty} (M; \bS_{n})$ which is given by the formula
\[
\eps(f) z (y)= \sum_{j=1}^n f_j (y) \eps_j  z.
\]
This is a family of symmetric, densely defined operators parametrized by $X$ (it is also essentially self-adjoint, which we do not need to know). 
We consider the operator
\[
B:= gDg \otimes 1 + \eta \otimes \eps(f)
\]
on the $\gA$-vector bundle $E \otimes \bS_{n}$. To understand this formula, note that $\bS_{n}$ is (by definition) a trivial vector bundle. For $s \in \Gamma_{cv} (M;E)$ and $z \in \bS_{n}$, the operator $B$ is given by the formula
\[
B(s \otimes z)= gDg s \otimes z + \eta s \otimes \sum_{j=1}^n f_j \eps_j z =   gDg s \otimes z + \sum_{j=1}^n f_j\eta s \otimes  \eps_j z.
\]
Precisely as in the proof of Lemma \ref{lem:selfadjointness-weighted-dirac}, it follows from \cite[Theorem 1.14]{JEIndex1} that $B$ is essentially self-adjoint (the point is that $\eta \otimes \eps(f)$ is of order $0$ and hence commutes with the multiplication by any function). 

\begin{lem}\label{lem:weighted-dirac-plus-clifford-is-fredholm}
The operator family $B$ is a Fredholm family, and even does have compact resolvent.
\end{lem}
\begin{proof}
We use \cite[Theorem 2.40]{JEIndex1}, and for that, we have to compute $B^2$. Let $s \in \Gamma_{cv} (M;V)$ and $z \in \bS_{n}$. Then 
\[
B^2 (s\otimes z)= B (gDg s \otimes z + \sum_{j=1}^n f_j\eta s \otimes  \eps_j z) = 
\]
\[
=gDg^2 Dg s \otimes z + \sum_{j=1}^n gDg f_j\eta s \otimes  \eps_j z + \sum_{i=1}^n f_i\eta gDg  s \otimes  \eps_j z+ \sum_{i,j=1}^n f_i f_j s \otimes \eps_i \eps_j z.
\]
The first summand is a nonnegative operator, namely $(gDg \otimes 1)^2$. The last summand is
\[
\sum_{j,i=1}^n  f_i f_j s\otimes \eps_i \eps_j z = \sum_j f_j^2 s \otimes z + \sum_{j<i} f_i f_j s \otimes (\eps_i \eps_j + \eps_j \eps_i) z = \norm{f}^2 s \otimes z.
\]
Because $\eta D + D \eta=0$, the middle two summands add up to
\[
\sum_{j=1}^n (gDg f_j\eta+f_j\eta gDg) s \otimes  \eps_j z  =  \sum_{j=1}^n [gDg, f_j]\eta  s \otimes  \eps_j z,
\]
so altogether, we obtain
\[
B^2 = (gDg \otimes 1)^2+ \norm{f}^2 + \sum_{j=1}^n [gDg,f_j]\eta \otimes \eps_j.
\]
By assumption, $\sum_{j=1}^n [gDg,f_j]\eta \otimes \eps_j$ is bounded (locally in $X$). We can restrict our attention to a subset of $X$ over which $\norm{\sum_{j=1}^n [gDg,f_j]\eta \otimes \eps_j} \leq C$, by \cite[Lemma 2.18]{JEIndex1}. Altogether, these computations prove that 
\[
B^2 \geq -C + \norm{f}^2,
\]
and since $\norm{f}^2: M \to \bR$ is fibrewise proper and bounded from below (i.e. coercive), \cite[Theorem 2.40]{JEIndex1} shows that $B$ is a Fredholm family with compact resolvent.
\end{proof}

We have ``consumed'' the $\Cl^{0,n}$-action in the definition of $B$, but the $\Cl^{n,0}$-action $e$ is still there. We observe that $B$ is $\Cl^{n,0}$-antilinear, because
\[
 B (\eta \otimes e(v)) + (\eta \otimes e(v)) B = 1 \otimes (e(v) \eps(f)+\eps(f) e (v)) + (\eta gDg+ gDg\eta) \otimes e(v)=0.
\]
Therefore, $(L^2_X (M;V), \eta \otimes \iota, e, B) \in \bK(\gA)_{n}(X)$, by the definition of $\bK(\gA)_{n}(X)$. The construction given is completely natural (since the auxiliary function $g$ was built into the definition of the sheaf $\GRW\gA_n^{op}$), and so this defines a map of sheaves
\[
\ind_n: \GRW\gA_n^{op} \to \bK ( \gA)_n,
\]
the \emph{analytical index}.

\begin{prop}\label{analytical-index-vs-scanning}
The collection $(\ind_n)_n$ is a weak map of spectra $\GRW \gA^{op} \to \bK (\gA)$ in the sense of Lemma \ref{lem:strictification-of-spectrummaps}. In other words, the diagram
\[
\xymatrix{
\GRW\gA^{op}_n \ar[r]^-{\scan}\ar[d]^{\ind_n} & \Omega \GRW \gA^{op}_{n+1} \ar[d]^{\Omega \ind_{n+1}}\\
\bK (\gA)_{n} \ar[r]^-{\bott} & \Omega \bK (\gA)_{n+1}
}
\]
commutes up to homotopy.
\end{prop}

\begin{proof}
Before we begin the proof, we emphasize that all definitions were designed so that this is essentially a tautology. Let $\gv:= (M, \pi,f,E,\eta,c,D,g) \in \GRW \gA^{op}_n (X)$. We will provide a natural (with respect to maps of test spaces) isomorphism between the cycles $\ind_{n+1}(\scan (\gv))$ and $\bott (\ind_n (\gv)) \in \Omega \bK(\gA)_{n+1}(X)$. This natural isomorphism then provides a natural concordance, by \cite[Lemma 3.6]{JEIndex1}. Let us first compute $\ind_{n+1}(\scan (\gv))$. By Lemma \ref{lem:grw-op-spectrudsf},
\[
\scan (\gv)=(\bR \times M, \id \times \pi, f',\pr_M^* E,D',g')\in \GRW \gA^{op}_{n+1} (\widehat{\bR} \times X ,\{\pm \infty\}\times X).
\]
The fibre $(\id \times \pi)^{-1} (t,x)$ is empty if $t=\pm \infty$ and equal to $\pi^{-1}(x)$ otherwise, and the restriction of $\pi^*_M E$ to $(\id\times \pi)^{-1}(t,x)$ coincides with $E|_{\pi^{-1}(x)}$ with all structures (Clifford structure, grading, Dirac operator and moderating function), and $f'(t,y)= f(y)+t e_{n+1}$.

According to the construction of the analytical index, $\ind_{n+1} (\scan (\gv))$ is represented by the following $K^{n+1,0}(\gA)$-cycle on $\bR\times X$ (extended by zero to $\widehat{\bR} \times X$):
\[
(\pr_X^* L^2_X (M;E \otimes \bS_{n+1}), \eta \otimes \iota_{n+1}, c \otimes e, g'D'g' \otimes 1+ \eta \otimes \eps(f')).
\]
On the other hand
\[
 \ind_n (\gv) = (L^2_X (M;E\otimes \bS_{n}), \eta \otimes \iota_n , c \otimes e, gDg + \eps(f)),
\]
and by the definition of the Bott map,
\[
\bott (\ind_n (\gv))= j_! \pr_X^* ( L^2_X (M;E\otimes \bS_{n})\otimes \bS_{1}, \eta \otimes \iota_n \otimes \iota_1, c \otimes e \otimes e,gDg+\eps(f)+ t \eps_{n+1})
\]
where $j: \bR \times X \to \widehat{\bR} \times X$ is the inclusion. Now we use that the operator family $g'D'g'$ is the same as the pullback of the original operator family $gDg$ along the projection map $\bR \times X\to X$, and we can write $\eps(f')$ at $(t,y )\in \bR \times M$ as $\eps (f(y))+ t \eps_{n+1}$. 
Moreover, under the natural isomorphism $\bS_{n+1}\cong\bS_{n} \otimes \bS_{1}$, we can write the grading $\iota=\iota_{n+1}=\iota_n \otimes \iota_1$ and 
\[
g'D'g' \otimes 1+ \eta \otimes \eps(f') = gDg \otimes 1 \otimes 1+ \eta \otimes \eps(f)\otimes 1 + \eta \otimes \iota_n \otimes t\eps_{n+1}.
\]
We obtain a natural isomorphism
\[
\bott (\ind_n (\gv))\cong j_! \pr_X^* ( L^2_X (M;E\otimes \bS_{n+1}), \eta \otimes \iota_{n+1} , c \otimes e,gDg+\eps(f)+ t \eps_{n+1})
\]
which finishes the proof.
\end{proof}

\section{The index theorem}\label{sec:index-theorem}

\subsection{Statement of the index theorem}\label{subsec:statement-index}

The results of the previous section can be summarized in a diagram
\[
 \xymatrix{
 \MT \gA  \ar[r]^-{\Lambda^{op}} & \GRW \gA^{op}  \ar[d]^{\xi} \ar[r]^-{\ind} & \bK (\gA)\\
  & \GRW \gA &
 }
\]
of spectra and (weak) spectrum maps (in the category $\Sheaf$). The map $\xi$ is a levelwise equivalence of spectra, by \ref{lem:forgetting-diff-data}, and $\Lambda^{op}$ is a stable equivalence of spectra, since the composition $\Lambda = \xi \circ \Lambda^{op}$ is, by Theorem \ref{grwtheorem1}. In Definition \ref{defn:topological-indexnasn}, we defined the topological index, a weak spectrum map $\topind: \MT \gA \to \bK(\gA)$. 
\begin{thm}[The index theorem]\label{thm:index-theorem}
For each $n\geq 0$, there is a homotopy $\ind_n \circ \Lambda_n^{op} \sim \topind_n: \MT \gA_n \to \bK(\gA)_n$.
\end{thm}

Let us now give a reformulation of the index theorem from which it becomes apparent that it generalizes the classical Atiyah-Singer theorem. Let $X \in \Mfds$ be of finite type (i.e., homotopy equivalent to a finite CW complex) and let $\gv \in \GRW \gA^{op}_n (X)$ be an element. Via the bijection \eqref{eqn:sheaf-representing}, $\gv$ gives rise to map $f_\gv: X \to |\GRW\gA^{op}_n|$, unique up to homotopy. We want to compute the composition $|\ind_n \circ f_\gv = f_{\ind_n (\gv)} : X \to |\bK(\gA)_n|$. Because $\tau_n : |\bK(\gA)_n| \to \Omega^{\infty-n} |\bK(\gA)|$ is a weak equivalence, we can equally ask for a computation of $\tau_n \circ |\ind_n| \circ f_\gv$.

Using the strictification procedure from Lemma \ref{lem:strictification-of-spectrummaps}, we obtain a spectrum map
\[
 \widetilde{|\ind|}: |\GRW\gA^{op}| \to |\bK(\gA)|.
\]
Furthermore, $|\topind|: |\MT \gA| \to |\bK(\gA)|$ is a weak spectrum map, and it also has a strictification $\widetilde{|\topind|}$. The map $|\Lambda^{op}|: |\MT \gA| \to |\GRW \gA|$ is already a spectrum map. 
Consider the diagram
\begin{equation}\label{eq:diagram-space-level-index-theorem}
\xymatrix{
\Omega^{\infty-n} |\MT \gA| \ar[rr]^{\Omega^{\infty-n} |\Lambda^{op}|}& &\Omega^{\infty-n} |\GRW\gA^{op}|\ar[rr]^{\Omega^{\infty-n} \widetilde{|\ind|}} & &\Omega^{\infty-n} |\bK(\gA)|\\
|\MT \gA_n| \ar[u]^{\tau_n} \ar[rr]^{|\Lambda^{op}_n|}& & |\GRW \gA^{op}_n| \ar[rr]^{|\ind_n|}\ar[u]^{\tau_n} & &|\bK(\gA)_n|\ar[u]^{\tau_n}_{\simeq}.
}
\end{equation}
The left square commutes for formal reason, and the right square commutes up to homotopy because $\widetilde{|\ind|}$ is a strictification of $(|\ind_n|)_n$. As in remark \ref{rem:pontrjaginthom}, we let $p_n: \Omega^{\infty-n} |\GRW \gA^{op}| \to \Omega^{\infty-n} |\MT \gA|$ be a homotopy inverse to $\Omega^{\infty-n} |\Lambda^{op}|$ and let $\PT_n := p_n \circ \tau_n$. It follows that
\[
\tau_n \circ |\ind_n| \circ f_\gv \sim (\Omega^{\infty-n} \widetilde{ |\ind|}) \circ \tau_n \circ f_\gv \sim (\Omega^{\infty-n} \widetilde{|\ind|}) \circ (\Omega^{\infty-n} |\Lambda^{op}|) \circ p_n \circ \tau_n\circ f_\gv  =
\]
\[
=(\Omega^{\infty-n} (\widetilde{|\ind|} \circ|\Lambda^{op}|) ) \circ \PT_n  \circ f_\gv .
\]
For each $n$, there are homotopies
\[
\widetilde{|\ind|}_n \circ |\Lambda_n^{op}| \sim |\ind_n| \circ |\Lambda^{op}_n| \stackrel{\eqref{thm:index-theorem}}{\sim}|\topind_n| \sim \widetilde{|\topind|}_n,
\]
by Theorem \ref{thm:index-theorem}. Hence we can apply Lemma \ref{lem:weakhomotopy-infinitelopopmap} to the spectrum maps $\widetilde{|\ind|} \circ |\Lambda^{op}|$ and $\widetilde{|\topind|}$. It follows that 
\[
(\Omega^{\infty-n} (\widetilde{|\ind|} \circ |\Lambda^{op}|) ) \circ \PT_n  \circ f_\gv \sim (\Omega^{\infty-n} \widetilde{|\topind|}  ) \circ \PT_n  \circ f_\gv .
\]
Since any finite CW complex is homotopy equivalent to a manifold of finite type, we conclude
\begin{corollary}\label{cor:indextheorem-space-version}
The two maps 
\[
\tau_n \circ |\ind_n|, \; (\Omega^{\infty-n} \widetilde{|\topind|}  ) \circ \PT_n:  |\GRW\gA_n^{op}| \to \Omega^{\infty-n} |\bK(\gA)| 
\]
are weakly homotopic. 
\end{corollary}

For $\gA=\gC$ and $n=0$, this is the version of the index theorem stated in \cite{EbertV} and is equivalent to the classical Atiyah-Singer family index theorem.

\subsection{The linear index theorem}

The proof of Theorem \ref{thm:index-theorem} has two parts. One is to compute the composition $\ind_n \circ \Lambda^{op}_n$ and rewrite the result in the form $\thom (\gy_n)$, where $\gy$ is a concretely given $\theta_\gA$-twisted $K(\gA)$-cycle on $\cC_\gA$. The other part is to prove that $\gy$ is naturally concordant to the cycle $\gx$ defined in Example \ref{ex:twisted-K-cycle-withoutoperator}. This step contains some substantial analytical arguments, and is carried out in this section. 

Let us begin with a classical and fairly elementary index computation. On the space $L^2 (\bR^d;\bS_d)$ of $L^2$-functions with values in the canonical Clifford module, we have the two (unbounded) operators $D$ and $F$ given by
\[
D = \sum_j e_j \partial_j ; \; F= \sum_j x^j \eps_j
\]
(here $x^j: \bR^d \to \bR$ denotes the $j$th coordinate function). Using our previous conventions, $F=\eps(\id_{\bR^d})$. The \emph{Bott-Dirac operator} or \emph{supersymmetric harmonic oscillator} is the operator $B=D+F$. 
We take $C^{\infty}_c (\bR^d; \bS_{d})$ as initial domain; $D$ and $F$ are symmetric on this domain.
The following result is more or less a standard result, see e.g. \cite[Proposition 1.16]{HigsonGuent}.

\begin{prop}\label{prop:index-supersymmetric-oscillator}\mbox{}
\begin{enumerate}
\item The operators $D$, $F$ and $B$ are formally self-adjoint, $O(d)$-equivariant and odd with respect to the grading $\iota$ on $\bS_{d}$.
\item The operator $B$ is essentially self-adjoint, and Fredholm.
\item The kernel of $B$ is $1$-dimensional, spanned by the even $O(n)$-invariant function $e^{-|x|^2 } 1$ (here $1 \in \Lambda^0 \bR^d\subset \bS_{d}$). Hence $\ind (B)=1\in KO^0_{O(d)} (*)$. 
\item Moreover $\spec (B) \cap (-1,1)=\{0\}$. 
\end{enumerate}
\end{prop}

In the context of this paper, the easiest way to the that $B$ is self-adjoint is to observe that for each linear form $\ell$ on $\bR^d$, $[B,\ell]$ is bounded and to use Lemma \ref{lem:selfadjointness-weighted-dirac}. The easiest way to see that $B$ is Fredholm is to observe that $B^2 = D^2+ |x|^2+ \nu$, where $\nu:=\sum_{i=1}^d e_i\eps_i $ obviously has norm $\leq d$, and to invoke \cite[Theorem 2.40]{JEIndex1}. 

Due to the $O(d)$-equivariance, the construction can be carried over to the parametrized case, when $\pi:V \to X$ is a Riemannian vector bundle (say $X$ is a manifold and $\pi$ is smooth). Consider $\bS_{V} \to X$, which is a fibrewise irreducible $\Cl (V \oplus V^-)$-module bundle. Denote the Clifford multiplication of $V$ by $e$ and that of $V^-$ by $\eps$. 
Then on the bundle $\pi^* \bS_{V} \to V$, we have families $D$, and $B$ of Dirac operators, parametrized by $X$. In a single fibre $V_x$, and with respect to an orthonormal basis $(v_1,\ldots,v_d)$ of $V_x$, they are defined by 
\[
D = \sum_{j=1}^d e(v_j) \partial_j; \; F = \sum_{j=1}^d x^j \eps (v_j).
\]
Because of the $O(d)$-equivariance, this does not depend on the choice of the orthonormal basis.

Now let $(E,\eta,c)$ be (smooth) bundle of finitely generated projective Hilbert-$\gA$-modules with graded $\Cl (V)$-structure. The bundle $\pi^* (E \otimes \bS_{V})\to V$ is a bundle of finitely generated projective Hilbert-$\gA$-modules. It has a grading $\eta \otimes \iota$ and a $\Cl(V \oplus V \oplus V^-)$-structure. The Clifford action by the first $V$-summand is by $c \otimes 1$, that by the second $V$-summand by $\eta \otimes e$, and that by the $V^-$-summand by $\eta \otimes \eps$. Let $D_E$ be the Dirac operator of the $\Cl(V)$-$\gA$-bundle $\pi^* E \to V$. On a single fibre over $x \in X$ and with respect to an orthonormal basis $(v_1, \ldots, v_d)$ of $V_x$, it is given by 
\[
D_E = \sum_{j=1}^d c(v_j) \partial_{j}. 
\] 
Now define a differential operator on $\pi^* (E \otimes \bS_{V})$ by 
\[
B_0:=D_E \otimes 1+ \eta \otimes F.
\]
This is an odd and symmetric unbounded operator family on the continuous field $L^2_X (V;\pi^* (E \otimes \bS_{V}))$.

\begin{lem}\label{lem:bott-dirac-regularity}
The operator family $B_0$ is an (unbounded) self-adjoint Fredholm family and it defines a $K^V(\gA)$-cycle on $X$:
\[
\gy(E, \eta,c):=(L^2_X (V;\pi^* (E \otimes \bS_{V})), \eta \otimes \iota, \eta \otimes e, B_0)\in \bK^{V} \gA(X).
\]
\end{lem}

Before we give the proof, let us state the main result of this subsection. Recall the element $\gx(E,\eta,c):= (E,\eta,c,0)$ defined in Example \ref{ex:twisted-K-cycle-withoutoperator}.

\begin{prop}[The linear index theorem]\label{prop:linear-index-theorem1}
There is a canonical concordance $\gy(E,\eta,c) \sim \gx(E,\eta,c)$ of $\theta_\gA$-twisted $K(\gA)$-cycles.
\end{prop}

\begin{proof}[Proof of Lemma \ref{lem:bott-dirac-regularity}]
One can prove this Lemma by directly verifying the hypotheses of \cite[Theorem 2.40]{JEIndex1}. Instead of doing this, we give an argument that will be used again in the proof of Proposition \ref{prop:linear-index-theorem1}. We transform $B_0$ by an isometric isomorphism of $L^2_X (V;\pi^* (E \otimes \bS_{V}))$ into an operator which looks more closely related to the Bott-Dirac operator. Namely, we define the operator $B_1:= \eta \otimes (D+F)$ on $\pi^* (E \otimes \bS_{V})$, using the Bott-Dirac operator $(D+F)$ on $\bS_{V}$. We claim that $B_0$ and $B_1$ are conjugate by an isometry. 

Let $x \in X$ and pick an orthonormal basis $(v_1, \ldots, v_d)$ of $V_x$. To ease notation, we denote by $c_j, e_j, \eps_j$ be the Clifford action of these basis vectors with respect to $c,e,\eps$ on the fibre $(E \otimes \bS_{V})_x$. Then
\begin{equation}\label{formula:harmonic-oscillators}
B_0= \sum_j c_j \partial_j + x^j \eps_j  \text{ and } B_1 = \sum_j e_j \partial_j + x^j \eps_j.
\end{equation}
Let 
\begin{equation}\label{defn:psi}
\psi := \exp (\frac{\pi}{4} \sum_{i=1}^{d} c_i e_i)= \frac{1}{2^{d/2}}\prod_{i=1}^{d} (1 -e_i c_i) \in \Cl (V\oplus V)_x \subset \Cl (V\oplus V \oplus V^-)_x.
\end{equation}
Then $\psi$ is even, $\psi^*\psi=1$, and the relations
\begin{equation}\label{psirelations}
 \psi e_j = -c_j \psi; \; \psi c_j = e_j \psi; \; \psi \eps_j = \eps_j \psi
\end{equation}
hold. Using \eqref{formula:harmonic-oscillators}, we get that
\begin{equation}\label{psi-intertwines-Di}
\psi B_0 \psi^{-1} = \sum_j \psi c_j \psi^{-1} \partial_j + x^j \eps_j =  \sum_j e_j \partial_j + x^j \eps_j =B_1.
\end{equation}
The element $\psi$ does not depend on the choice of the orthonormal basis of $V_x$ (the quickest way to prove this is: observe that rotation in the $v_1-v_2$-plane does not change $\psi$ and neither does permutation of basis vectors, and use that these rotations and permutations generate the orthogonal group $O(V)$). Therefore $\psi$ gives a global isometry of $\pi^* (E \otimes \bS_{V})$.  

Therefore, it is enough to prove that $B_1$ is a self-adjoint Fredholm family. Self-adjointness is proven as in Lemma \ref{lem:selfadjointness-weighted-dirac}, using \cite[Theorem 1.14]{JEIndex1}.
For the Fredholm property, compute
\[
B_1^2 = 1 \otimes (D+F)^2 = 1 \otimes (D^2 + |v|^2 + \nu) \geq |v|^2- d.
\]
Therefore $B_1^2$ is bounded from below by the coercive function $|v|^2-d: V \to \bR$.
By \cite[Theorem 2.40]{JEIndex1}, $B_1$ is Fredholm.
\end{proof}

\begin{proof}[Proof of Proposition \ref{prop:linear-index-theorem1}]
In the first step, we use the isometry $\psi$ from the proof of Lemma \ref{lem:bott-dirac-regularity}: 
\[
\begin{split}
\gy(E, \eta,c)=(L^2_X (V;\pi^* (E \otimes \bS_{V})), \eta \otimes \iota, \eta \otimes e, B_0) \cong\\
(L^2_X (V;\pi^* (E \otimes \bS_{V})), \eta \otimes \iota, -c \otimes 1, B_1)=\\
(L^2_X (V;\pi^* (E \otimes \bS_{V})), \eta \otimes \iota, -c \otimes 1, \eta \otimes (D+F))
\end{split}
\]
(the minus sign in front of $-c\otimes 1$ comes from the relations \eqref{psirelations}). 

Now let $p_0 \in \Kom_{X} (L^2_{\pi} (V; \pi^*\bS_{V}))$ be the projector onto $\ker (D+F)$ (which is a rank $1$ trivial vector bundle on $X$, by Proposition \ref{prop:index-supersymmetric-oscillator} and equivariance) and let 
\[
p=1\otimes p_0\in \Kom_{X, \gA} (L^2_{\pi} (V; \pi^*(E\otimes \bS_{V})))
\]
(to see that $p$ is compact, one uses that $p_0$ is a rank $1$ operator, globally, and applies the definition of a compact operator family \cite[Definition 2.15]{JEIndex1}). Then $p$ is a projection and commutes with the grading, Clifford structure and with $\eta \otimes (D+F)$. Hence we get an equality
\[
\begin{split}
(L^2_{X} (V; \pi^*( E \otimes \bS_{V})), \eta \otimes \iota ,-c \otimes 1 , \eta \otimes (D+F)) =\\
=(\im (p), \eta \otimes \iota ,-c \otimes 1 , (\eta \otimes (D+F))|_{\im (p)}) \oplus (\im (1-p), \eta \otimes \iota ,-c \otimes 1 , (\eta \otimes (D+F))|_{\im (1-p)})=\\
=(E, \eta  ,-c  , 0) \oplus (\im (1-p), \eta \otimes \iota ,-c \otimes 1 , (\eta \otimes (D+F))|_{\im (1-p)})\sim (E, \eta  ,-c  , 0).
\end{split}
\]
The second summand is degenerate, since $((\eta \otimes (D+F))|_{\im (1-p)})^2 \geq 1$ by Proposition \ref{prop:index-supersymmetric-oscillator} (4), and is hence (canonically) concordant to the zero cycle, by \cite[Lemma 3.9]{JEIndex1}. The first summand is isomorphic to $(E, \eta , c ,0)$, via $\eta$, and hence canonically concordant to that cycle.
\end{proof}

\subsection{Proof of the index theorem}

Recall that the topological index $\topind_n: \MT \gA_n \to \bK(\gA)_n$ is defined as $\topind_n = \thom (\gx_n)$. By Proposition \ref{prop:linear-index-theorem1} and Lemma \ref{lem:thom-class-is-spectrum-map}, we therefore have
\[
 \topind_n \sim \thom (\gy_n).
\]
Therefore, to complete the proof of Theorem \ref{thm:index-theorem}, we need to prove the following result.

\begin{lem}
There is a natural concordance $\thom (\gy)_n \sim \ind_n \circ \Lambda_n^{op}$ of maps $\MT \gA_n \to \bK (\gA)_n$ of sheaves.
\end{lem}

\begin{proof}
As in the proof of Proposition \ref{analytical-index-vs-scanning}, the canonical concordance will be given by a natural isomorphism. Let $X$ be a test manifold and $\gv:=(U,V,\pi,E,\eta,c,s) \in \MT \gA_n(X)$. Recall that $U \subset X$ is open, $\pi:V \subset U \times \bR^n\to U$ is a rank $d$ vector bundle, $(E,\eta,c)\to U$ a bundle of finitely generated projective Hilbert-$\gA$-modules with grading $\eta$ and $\Cl(V)$-action $c$. Finally, $s: U \to V^{\bot}$ is a smooth section with the growth condition (i.e. if $x_n \in U$ converges to $x \in \overline{U} \setminus U$, then $\norm{s(x_n)} \to \infty$).

Let us first compute $\ind_n (\Lambda^{op}_n (\gv))$. The map $\Lambda^{op}_n : \MT \gA_n(X) \to \GRW \gA^{op}_n(X)$ constructed in Lemma \ref{lem:from-linear-mfds-to-dirac} assigns to $\gv$ the element
\[
(V,\pi,f,\pi^* E,\eta,c,D_E,1)\in \GRW\gA^{op}_n (X),
\]
with $\pi: V \to U \subset X$ and $f(x,v):= v+s(x) \in \bR^n$; and $D_E$ is the Dirac operator. Hence
\begin{equation}\label{eqn:indexproof1}
\ind_n (\Lambda_n^{op} (\gv))=(L^2_X (V;\pi^* E \otimes \bS_{n}), \eta \otimes \iota_{\bR^n}, \eta \otimes e_{\bR^n}, D_E \otimes 1+ \eta \otimes \eps(f))\in \bK(\gA)_n(X).
\end{equation}
The submersion $\pi: V\to X$ is the composition of the bundle projection $\pi:V \to U$ with the inclusion map $i: U \to X$. We can rewrite the formula (\ref{eqn:indexproof1}) as 
\begin{equation}\label{eqn:indexproof2}
\ind_n (\Lambda_n^{op} (\gv))=i_!(L^2_U (V;\pi^* E \otimes \bS_{n}), \eta \otimes \iota_{\bR^n}, \eta \otimes e_{\bR^n}, D_E \otimes 1+ \eta \otimes \eps(f))\in \bK (\gA)_n(X).
\end{equation}
Now use the canonical isomorphism $\bS_{n} \cong  \bS_{V} \otimes \bS_{V^\bot}$ (of bundles over $U$), and that $\iota_{\bR^n} = \iota_V \otimes \iota_{V^\bot}$, $e_{\bR^n} = e_V \otimes 1 + \iota_V \otimes e_{V^\bot}$ under this isomorphism. Because $f(x,v):= v+s(x)$, we have $\eps (f)= \eps(\id_V) \otimes 1 + \iota_V \otimes \eps(s)$. Altogether, the right-hand side of (\ref{eqn:indexproof2}) becomes
\[
i_!(L^2_U (V;\pi^* E \otimes \bS_{V} \otimes \bS_{V^\bot} ), \eta \otimes \iota_V \otimes \iota_{V^\bot}, \eta \otimes e_V \otimes e_{V^\bot}, D_E \otimes 1 \otimes 1+ \eta \otimes \eps(\id_V) \otimes 1+ \eta \otimes \iota_V \otimes \eps(s)),
\]
and this is canonically isomorphic to 
\[
i_!(L^2_U (V;\pi^* E \otimes \bS_{V}  )\otimes \bS_{V^\bot}, (\eta \otimes \iota_V) \otimes \iota_{V^\bot}, (\eta \otimes e_V) \otimes e_{V^\bot}, (D_E \otimes 1+ \eta \otimes \eps(\id_V)) \otimes 1+ (\eta \otimes \iota_V )\otimes \eps(s)).
\]
By the formula for the Thom homomorphism\eqref{thom-iso-as-sheaf-map2}, this is the same as $\thom (\gy_n)(\gv)$, as claimed.
\end{proof}

\bibliographystyle{plain}
\bibliography{index}
\end{document}